\documentclass[10pt]{amsart}
\usepackage{amsmath,amsfonts,amssymb,latexsym,amsthm}
\usepackage{epsfig}
\usepackage{hyperref} % to make links for lemmas and theorems
\hypersetup{colorlinks} %so the links look nicer
\usepackage[dvipsnames]{xcolor}

\usepackage{dsfont}
\usepackage{mathtools}
\usepackage{fancyhdr}
\usepackage{enumitem}
\usepackage{pgfplots}
\usepackage{url}
\usepackage{tcolorbox}
\usetikzlibrary{graphs}
\usetikzlibrary{graphs.standard}
\usetikzlibrary{arrows}
\usepackage{xcolor}

\usepackage{enumitem}

\numberwithin{equation}{section}

\newtheorem{thm}{Theorem}[section]
\newtheorem{lemma}[thm]{Lemma}

\newtheorem{cor}[thm]{Corollary}

\newtheorem{conj}[thm]{Conjecture}
\newtheorem{conv}[thm]{Convention}

\newtheorem{rem}[thm]{Remark}

\newtheorem{theorem}[thm]{Theorem}

\newtheorem{corollary}[thm]{Corollary}%[section]
\newtheorem{sublemma}[thm]{Sublemma}
\newtheorem{remark}[thm]{Remark}

\makeatletter
\newcommand{\eqnum}{\leavevmode\hfill\refstepcounter{equation}\textup{\tagform@{\theequation}}}\makeatother

\newcommand{\enn}{\mathbb{N}}
\newcommand{\zee}{\mathbb{Z}}

\DeclareMathOperator{\cog}{cog}
\DeclareMathOperator{\Cog}{Cog}

\def\bba{{\textbf{\textit{a}}}}
\def\bbb{{\textbf{\textit{b}}}}

\def\bx{{\textbf{\textit{x}}}}
\def\by{{\textbf{\textit{y}}}}
\def\bd{{\textbf{\textit{d}}}}
\def\bbr{{\textbf{\textit{r}}}}

\newcommand{\red}[1]{{\bf \textcolor{red}{#1}}}
\newcommand{\blu}[1]{{\bf \textcolor{blue}{#1}}}
\newcommand{\xii}{\xi}

\title[Cogrowth sequence of nilpotent groups]{Algebraic and arithmetic properties of the cogrowth sequence of nilpotent groups}

\author[Igor Pak \, and \, David Soukup]{ Igor Pak$^\star$ \, and \, David Soukup$^\star$}

\thanks{\today}
\thanks{\thinspace ${\hspace{-.45ex}}^\star$Department of Mathematics,
UCLA, Los Angeles, CA~90095.
\hskip.06cm
Email:
\hskip.06cm
\texttt{\{pak,soukup\}@math.ucla.edu}}

\newcommand{\bin}{\mathrm{bin}}
\newcommand{\diag}{\mathrm{diag}}

\newcommand{\Tow}{\mathrm{Tow}}

\def\zz{\mathbb Z}
\def\nn{\mathbb N}

\def\rr{\mathbb R}

\def\Ga{\Gamma}

\def\al{\alpha}
\def\be{\beta}

\def\CR{\mathcal R}

\def\CS{\mathcal S}

\def\ce{\mathcal E}

\def\CT{\mathcal T}

\def\ssu{\subset}

\def\wt{\widetilde}
\def\<{\langle}
\def\>{\rangle}

\def\Deg{{\small \text{\rm D}}}

\def\0{{\mathbf 0}}

\def\.{\hskip.06cm}
\def\ts{\hskip.03cm}

\def\nin{\noindent}

\DeclareMathOperator{\SL}{SL}
\DeclareMathOperator{\UT}{UT}
\DeclareMathOperator{\PSL}{PSL}

\usepackage[colorinlistoftodos]{todonotes}

\RequirePackage{cleveref}
\usepackage{hypcap}
\hypersetup{colorlinks=true, citecolor=darkblue, linkcolor=darkblue}
\definecolor{darkblue}{rgb}{0.0,0,0.7}
\newcommand{\darkblue}{\color{darkblue}}

\definecolor{darkred}{rgb}{0.68,0,0}

\definecolor{darkgreen}{rgb}{0,.38,0}
\newcommand{\darkgreen}{\color{darkgreen}}

\newcommand{\defn}[1]{\emph{\darkblue #1}}

\newcommand{\defng}[1]{\emph{\darkgreen #1}}

\begin{document}

\begin{abstract}
We prove that congruences of the cogrowth sequence in a unitriangular
group \ts $\UT(m,\zz)$ \ts are undecidable.  This is in contrast with
abelian groups, where the congruences of the cogrowth sequence are decidable.
As an application, we conclude that there is no algorithm to present
the cogrowth series as the diagonal of a rational function.
\end{abstract}

\maketitle

%\vskip1.6cm
\section{Introduction}\label{s:intro}
On a fundamental level, the \emph{growth} and \emph{cogrowth sequences} are used
to extract global properties of finitely generated groups from a local
information.  Although many problems remain unresolved, the \emph{asymptotic
approach} to both sequences has led to a number of spectacular advances
(see below).

The \emph{algebraic approach} to growth and cogrowth sequences is usually stated
in terms of their generating functions (GF).  Do they satisfy an algebraic
equation?  What about a differential-algebraic equation?  Given that both
sequences are sensitive with respect to the change in the generating sets,
one might not think there is much to this problem, and yet there is a
plethora of  positive results and some notable negative results in this
direction (see below).

In this paper we present an arithmetic approach to the cogrowth sequences
of nilpotent groups as a means to obtain negative results for their algebraic
properties.  We first state the main results and historical remarks.
We postpone the applications until Section~\ref{s:series}.

\subsection{Main results} \label{ss:intro-main}
Let $G$ be a fixed finitely generated group, and let $\CS=\CS^{-1}$ be a
symmetric generating set $\<\CS\>=G$. Denote by \ts
$$\cog_\CS(n) \. := \. \big|\big\{(s_1,\ldots,s_n) \in \CS^n \.:\. s_1\cdots s_n=1\big\}\big|\ts.
$$
the number of products of generators equal to one.  The sequence $\{\cog_\CS(n)\}$ is called
the \defn{cogrowth sequence}.  It can be viewed as
the number of closed walks of length~$n$ in the  \emph{Cayley graph} \ts $\Ga(G,\CS)$.
The \defn{unitriangular group} \ts $\UT(m,\zz)$ \ts is the (nilpotent) group of $m\times m$ upper
triangular % integer
matrices with $1$'s on the diagonal.

\smallskip

\begin{theorem}[Main theorem] \label{t:maintheorem}
There exist integers \ts $m\ge 3$, \ts $a\ge 1$, and a prime \ts $p$, such that
the following problem is \underline{undecidable}\ts: \ts
Given symmetric generating sets \ts $\mathcal{S},  \ts \mathcal{T}$ \ts in \ts $\UT(m,\zee)$, determine whether
\[
\forall  \ts n\in \nn \ : \
\cog_{\mathcal{S}}(n) \ts \equiv \ts \cog_{\mathcal{T}}(n) \ \mod p^a\ts.
\]
Moreover, the result holds for \ts $p=2$, \ts $a=40$, and some \ts $m \leq 9.6 \cdot 10^{85}$.
\end{theorem}

\smallskip

This is a rare undecidable problem for the relatively tame class of nilpotent groups.
The proof uses a technical yet explicit embedding of general Diophantine equations
into the cogrowth. Solvability of Diophantine equations is famously
undecidable by the negative solution of \emph{Hilbert's 10th problem} (the
Matiyasevich, Robinson, Davis and Putnam theorem), see e.g.\ \cite{Mat93}.

\smallskip

Our main theorem should be compared with the following result:

\smallskip

\begin{thm}\label{t:abelian}
Let $a\ge 1$ be an integer, let \ts $p$ \ts be a prime, and let
\ts $G$ \ts be a finitely generated abelian group.
The following problem is \underline{decidable}\ts: \ts Given finite
symmetric generating sets \ts $\mathcal{S},  \ts \mathcal{T}$ \ts in \ts $G$,
determine whether
\[
\forall  \ts n\in \nn \ : \
\cog_{\mathcal{S}}(n) \ts \equiv \ts \cog_{\mathcal{T}}(n) \ \mod p^a\ts.
\]
\end{thm}

\smallskip

This result is derived from a remarkable theorem of Adamczewski and Bell
\cite{AB13}, which in turn extends a series of results by Furstenberg~\cite{Fur67},
Deligne \cite{Del84}, Denef and  Lipshitz~\cite{DL87}, on diagonals
of rational functions modulo prime powers.  Our own motivation for the main theorem comes from the opposite direction, and can be stated as follows.

The \defn{cogrowth series} \ts for the group $G=\<\CS\>$ is defined as
$$
\Cog_\CS(t) \, := \, 1 \. + \. \sum_{n=1}^{\infty} \. \cog_\CS(n) \ts t^n\ts.
$$
Let
$$B(x_1,\ldots,x_k) \, = \, \sum_{(n_1,\ldots,n_k)\in \nn^k} \.
b(n_1,\ldots,n_k) \, x_1^{n_1}\cdots x_1^{n_k} \. \in \. \zz[[x_1,\ldots,x_k]]
$$
be a multivariate generating function. The \defn{diagonal} \ts of $B$ is defined as \ts
$\sum_{n\ge 0} \. b(n,\ldots,n) \ts t^n$.

\smallskip

\begin{thm}  \label{t:diag}
For a fixed sufficiently large integer \ts $m$, the following problem is \ts
\underline{not computable}: \ts
Given a symmetric generating set \ts $\CS$ \ts of the unitriangular
group \ts $\UT(m,\zz)$, write the cogrowth series \ts $\Cog_\CS(t)$ \ts
as a diagonal of a rational function \ts $P/Q$,
for some polynomials \ts $P,Q\in \zz[x_1,\ldots,x_k]$, and \ts $k\ge 1$.
Moreover, the result holds for some \ts $m \leq 9.6 \cdot 10^{85}$.
\end{thm}

%\footnote{We chose ZFC to make the statement
% more accessible.  The proof naturally extends to any system of axioms.}

\smallskip

In other words, either some cogrowth series are not diagonal, or all of them
are diagonals, but the proof of that result would be ineffective to make the
diagonals uncomputable.  Let us mention a quick motivation for this problem
(see more on this below).

\defn{Kontsevich's question}, \ts  for the case of nilpotent groups (see below),
asks whether the cogrowth series \ts $\Cog_\CS(t)$ \ts is always \emph{D-finite},
i.e.\ a solution of an ODE with polynomial coefficients.  Christol's
Conjecture~\ref{conj:Christol} (see below), reduces the problem to
whether \ts $\Cog_\CS(t)$ \ts is always a diagonal of a rational function.
Until Theorem~\ref{t:diag}, no progress has been made in this direction.

\smallskip

\begin{rem}{\rm
Let us further discuss our Theorem~\ref{t:diag} in context of Kontsevich's
question. First, it is possible and even likely, that already for the
\emph{Heisenberg group} \ts $\UT(3,\zz)$ \ts with four standard generators,
the cogrowth series is not a diagonal (and non-D-finite), see~$\S$\ref{ss:finrem-Heis}.
It is also possible and even likely, that \emph{for all} \ts $m\ge 3$, and \emph{all}
symmetric generating sets \ts $\CS$ \ts of \ts $\UT(m,\zz)$, the cogrowth series
is not a diagonal.  Theorem~\ref{t:diag} gives no contradiction with that.

On the other hand, it is possible that \emph{for some} \ts $\CS$ \ts the
cogrowth series is a diagonal.  It is also possible that \emph{for all} \ts $\CS$ \ts
the cogrowth series is a diagonal.  What Theorem~\ref{t:diag} shows is that there is
no \defng{constructive proof} \ts that the cogrowth series it is
always a diagonal.}
\end{rem}

\smallskip

\subsection{Historical background}\label{ss:intro-hist}
Here we give a very brief overview of the vast literature on the subject.

\smallskip

\nin
{\small $(1)$} \.
The growth of groups goes back to the works of Schwarz (1955) and
Milnor (1968), and is now a staple of Geometric Group Theory \cite{Har21}.
Notably, all nonamenable groups have exponential growth, but not vice versa.
\emph{Gromov's theorem} \ts proves that the growth is polynomial if and only if the
group is virtually nilpotent.  We refer to \cite[Ch.~VI,\ts VII]{Har00}
for an extensive introduction, and to \cite{Mann} for a detailed treatment.

In probabilistic context, the cogrowth was first introduced by P\'olya \cite{Polya},
to study transience and recurrence of random walks in $\zz^d$, via
asymptotic estimates on the \defng{return probability} \ts $\cog_\CS(n)/|\CS|^n$,
and later by Kesten~\cite{Kes59} in connection with \emph{amenability}.
In Group Theory, the study of cogrowth was initiated by Grigorchuk \cite{Gri80}
and extended by Cohen \cite{Coh82} and others.  We refer to \cite{Woe00}
for a comprehensive presentation of both group theoretic and probabilistic results.

\smallskip

\nin
{\small $(2)$} \.
The generating function (GF) approach became popular after the
\emph{Golod--Shafarevich theorem} on the growth of algebras \cite[$\S$3.5]{Ufn95}.
In a remarkable development, the \defng{growth series} (the GF for the growth
sequence) is shown to be rational for every generating set of many classes of groups,
including virtually abelian~\cite{Ben83} and hyperbolic~\cite{Can84}.

For other classes of groups, growth series can be more complicated.
Notably, there are wreath products of abelian groups for which growth series are algebraic
but not rational \cite{Par92}.  For the  fundamental group of a $3$-dimensional
\ts $\PSL(2,\rr)$-manifold,
which is a $\zz$-extension of a hyperbolic group, the growth series is rational
for one generating set and non-algebraic for another \cite{Sha94}.
It is known (see e.g.\ \cite{GP17}) that the growth series is non-algebraic
(in fact, non-D-finite), for all groups of intermediate growth.
See \cite[$\S$4]{GH97} for further examples and many references.

For nilpotent groups, the growth series is especially interesting.
In a breakthrough paper \cite{Sto96},  Stoll gave an example of a
\emph{higher Heisenberg group} \ts $H_2 \ssu \UT(4,\zz)$ \ts and two
generating sets so that one growth series is rational while another
is non-algebraic.  Curiously, for the (usual) Heisenberg group
\ts $H_1 = \UT(3,\zz)$, the growth series is always rational \cite{DS19}.

\smallskip

\nin
{\small $(3)$} \.
After  P\'olya's work, \emph{lattice walks} \ts on $\zz^d$ have been
intensely studied for various generating sets~$\CS$ (called \emph{steps}).
The corresponding return probabilities are always diagonals of rational
functions, but this stops being true when  geometric constraints are added.
These walks continue to be intensely studied in Enumerative and
Asymptotic Combinatorics, see e.g.\ \cite{Bou06,Mis20}.

For free groups $F_k$, the cogrowth series are always algebraic. This was
shown independently in \cite{Hai93} in a combinatorial context,
and in \cite{Aom84,FTS94} in a probabilistic context. The cogrowth series
is algebraic for many free products of groups \cite{BM20,Kuk2},
and D-finite for Baumslag--Solitar groups \ts BS$(N,N)$ \cite{ERRW}.

In recent years, the interest to the problem came from Kontsevich's question
whether the cogrowth series is always D-finite on linear groups, see \cite{Sta14}.
By the \emph{Tits alternative} and the \emph{Milnor--Wolf theorem},
Kontsevich's question is reduced to three cases: \ts virtually nilpotent
groups, virtually solvable groups of exponential growth, and groups containing
free group $F_2$ as a subgroup.  Our state of knowledge is very different
in these three cases.

For solvable groups the question was resolved in the negative in~\cite{GP17} by
the following argument.
Let $G$ be a solvable group of exponential growth and bounded Pr\"ufer rank.
It was proved by Pittet and Saloff-Coste in \cite{PS03}, that for
every symmetric generating set $\CS$, the cogrowth satisfies
$$
|\CS|^n \ts e^{-\al n^{1/3}} \. \le \. \cog_\CS(n) \. \le \. |\CS|^n \ts e^{-\be n^{1/3}}\..
$$
The \emph{Birkhoff--Trjitzinsky theorem}\footnote{There are gaps in the proof
of this result and it remains an open problem in full generality,
see a discussion in \cite[$\S$VIII.7]{FS09} and \cite[$\S$9.2]{Odl95}. For
integral sequences which grow at most exponentially, the gaps were filled in
a series of paper, see \cite[$\S$5.1]{GP17}.}
then implies that the cogrowth series not D-finite \cite{GP17}.
An easy example of such group is \ts $\zz\ltimes \zz^2\ssu \SL(3,\zz)$, see
e.g.\ \cite[$\S$15.B]{Woe00}.
In response to a solution in \cite{GP17}, Katzarkov, Kontsevich and Stanley
independently asked if the cogrowth series is always D-algebraic.\footnote{Personal
communication, 2015.}  This strengthening of Kontsevich's question
remains unresolved.

In fact, the bounded Pr\"ufer rank assumption above is not necessary for the
conclusion.
Recently, Bell and Mishna used an analytic argument \cite{BM20} to show
that, for all amenable groups of superpolynomial growth, the cogrowth series
is non-D-finite, resolving the conjecture in~\cite{GP17} and completing this case
of Kontsevich's question.

For nilpotent groups, the subject of this paper,
the \emph{Bass--Guivarc'h formula} \ts computes the polynomial degree \ts
$d(G)$ \ts of the growth sequence.  Several notable probabilistic results
can be combined to give the following asymptotics
$$C_1 \ts |\CS|^n \ts n^{-d(G)/2} \. \le \. \cog_{\CS}(n) \. \le \. C_2 \ts |\CS|^n \ts n^{-d(G)/2},
$$
see \cite[$\S$3.B,$\ts\S$15.B]{Woe00} and references therein. Now
\emph{Jungen's theorem} \cite{Jun31}, implies that the cogrowth series
is not algebraic for even~$d(G)$.  For odd \ts $d(G)\ge 5$, only a weaker
result is known, that the cogrowth series is not $\rr_+$-algebraic; this follows
from \cite[Thm~3]{BD15}.  At this point the analytic arguments lose their
power as there are numerous examples of D-finite and even algebraic GFs with the
same asymptotics as the cogrowth sequences, see e.g.\ \cite{BD15,FS09}.

\smallskip

\nin
{\small $(4)$} \. \emph{Hilbert's 10th problem} was resolved by Matiyasevich (1970)
building on the earlier work by Davis, Putnam and Robinson (1949--1969).  Solvability
of Diophantine equations over various rings is now fundamental in both Logic
and Number Theory, and applied throughout mathematical sciences, from
Group Theory to Integral Programming.
We refer to \cite{Mat93} for a thorough treatment, to \cite{Poo08} for
a short note introduction to recent developments, and to \cite{MF19}
for an introductory textbook.

\smallskip

\nin
{\small $(5)$} \.
The study of classes of GFs was initially motivated by
applications in Number Theory and Analysis, but came to prominence
in connection to Formal Languages Theory.  The GF for the number of
accepted paths by a \emph{Finite State Automaton} is always rational
(see e.g.\ \cite[$\S$4.7]{Sta99}), and algebraic for a \emph{Pushdown
Automaton} (see references in~\cite{BD15}).

The class of diagonals of rational functions coincides with the class
of GFs for (balanced) \defng{binomial sums}, see \cite{BLS17,Gar09}.
This class received much attention after the work of Wilf and Zeilberger on
binomial identities \cite{WZ92,Zei90}, which made heavy use of the
fact that they are D-finite (\emph{holonomic} in their terminology).

Finding an explicit presentation of a GF as a diagonal of a rational
function is of great interest in Computer Algebra due to its
many applications, see e.g.\ \cite{BLS17,Mel21}.  These range from
congruences of combinatorial sequences, see \cite{AB13,RY15}, to
asymptotic analysis, see \cite{BMPS18,MS21}.  We should note that there
can be more than one way a function can be presented as diagonal,
see e.g.~\cite{RY15}.  On the other hand, for many series finding
its presentation as a diagonal is a challenging open problem,
see~$\S$\ref{ss:finrem-Christol}. Our Theorem~\ref{t:diag} proving
uncomputability of such presentation is the first negative result
in this direction.

Proving that a series is not D-finite
(not D-algebraic) is a major challenge, of interest both in Enumerative
Combinatorics \cite{Pak18} and Differential Algebra \cite{ADH17}.
Outside of analytic arguments, an Automata Theory approach was developed
in \cite{GP15}, which proves non-D-finiteness for GFs of various permutation classes.
In the context of cogrowth series, \cite{GP17} uses this approach to
prove non-D-finiteness in the (less interesting) case of
\emph{non-symmetric} \ts generating sets of nonamenable groups.

\smallskip

\nin
{\small $(6)$} \.
The undecidability approach to algebraic properties of cogrowth series
appears to be new.  It is also surprising, since both the
word, the conjugacy and even the isomorphism problems are
decidable for finite nilpotent groups \cite{GS80} (see also discussion
in \cite[$\S$3.2]{Sap11}).  On the other hand, the
solvability of a \emph{system of equations} \ts is undecidable
for \ts $H_1=\UT(3,\zz)$ \ts \cite{DLS15,GMO20}, as well group
membership in the product of cyclic subgroups of \ts $\UT(m,\zz)$ \ts 
\cite{Loh15}. The proofs of these results are similarly based on 
Hilbert's 10th problem, cf.~$\S$\ref{ss:finrem-Heis}.

\smallskip

\subsection{Paper structure}  \label{ss:intro-structure}
After a few notation in Section~\ref{s:notation},
we start with a technology of generating functions in Section~\ref{s:series}.
There, we give quick proofs of Theorem~\ref{t:abelian} from the Adamczewski--Bell
theorem (Theorem~\ref{t:AB}), and of Theorem~\ref{t:diag} from the
Main Theorem~\ref{t:maintheorem}.   There, we also formulate
Theorem~\ref{t:dalg} on a possible non-D-algebraic cogrowth series
for $\UT(m,\zz)$.
We then prove Main Theorem~\ref{t:maintheorem} in a lengthy Section~\ref{s:proof-main}.
The proof of Theorem~\ref{t:dalg} is given in Section~\ref{s:proof-dalg}.
We conclude with final remarks and open problems in Section~\ref{s:finrem}.

\medskip

\section{Notation}\label{s:notation}
We use the convention that \textbf{bold} letters represent multi-indices,
e.g.\ $\bx = (x_1,\ldots,x_k)\in \zz^k$.  We use \ts $|\bx|:=|x_1|+\ldots + |x_k|$ \ts
to denote the \ts $\ell^1$ \ts norm of~$\bx$. % No other norms will be used.

For vectors \ts $\bba,\bbb\in\zee^k$,
denote
\begin{equation}\label{multiindexbimonal}
\binom{\bba}{\bbb} \, := \, \binom{a_1}{b_1} \. \cdots \. \binom{a_k}{b_k}\ts.
\end{equation}

The unipotent group $\UT(m,\zee)$ is the group of all $m\times m$
upper-triangular integer matrices with ones on the diagonal:
{\small
\[
\begin{bmatrix}
1 & \zz & \zz & \cdots & \zz & \zz \\
0 & 1 & \zz & \cdots & \zz & \zz \\
0 & 0 & 1 & \cdots & \zz & \zz \\
\vdots & \vdots & \vdots &\ddots & \vdots & \vdots \\
0 & 0 & 0 & \cdots & 1 & \zz \\
0 & 0 & 0 & \cdots & 0 & 1 \\
\end{bmatrix}
\]
}

Since we will be working with many families of indexed matrices, we will adopt the convention that $[A]_{ij}$ refers to the $(i, j)$-th entry of matrix~$A$. Let $I_n$ be the $n\times n$ identity matrix, and $E_{i, j}$ be the matrix that is 1 in the $(i, j)$-th coordinate and $0$ otherwise.

When working with matrices, we write \ts $XY$ \ts to denote the product of matrices $X$ and~$Y$.
We use \ts $X \circ Y$ \ts to denote the word with matrices as letters. Lastly, we  use $\oplus$ for the operation of making a block-diagonal matrix out of smaller matrices:
\[
X \oplus Y \, := \, \begin{bmatrix}
X & 0 \\
0 & Y
\end{bmatrix} \,.
\]
We use \ts $X\oplus^k Y$ \ts to mean that $Y$ is added $k$ times: \ts $X\oplus Y \oplus \cdots \oplus Y$.
Finally, a word $(s_1\cdots s_n)$ in the generators $s_i\in \CS$, is called a \emph{cogrowth word},
if the product \ts $s_1\cdots s_n=1$.

\medskip

\section{Cogrowth series} \label{s:series}

\subsection{Classes of generating functions} \label{ss:series-classes}
Let $\{a_n\}$ be an integer sequence, and let
$$A(t) \, := \, \sum_{n=0}^{\infty} \. a_n \ts t^n \. \in \ts \zz[[t]]
$$
be the corresponding \emph{generating function} (GF).  We write $a_n=[t^n] A$
to denote the coefficient of the GF.   For a multivariate GF
$B\in \zz[[x_1,\ldots,x_k]]$, the \emph{diagonal} \ts of~$B$ is defined as
$$
\diag \ts B  \, := \, \sum_{n=0}^{\infty} \. \big(\big[x_1^n\cdots x_k^n\big] \ts B\big) \ts t^n \. \in \ts \zz[[t]]\ts,
$$
the GF for diagonal coefficients of~$B$.

\medskip

For $A\in \zz[[t]]$, we define the following five main classes of GFs, see e.g.\ \cite[Ch.~6]{Sta99}:

\medskip

\nin
\qquad {\it \textbf{Rational}}: \qquad $A(t) =P(t)/Q(t)$, \ for some \. $P,\ts Q\in \zz[t]$,

\smallskip

\nin
\qquad {\it \textbf{Algebraic}}: \ \quad \.\ts $c_0 A^k \ts + \ts c_1 A^{k-1}\ts+\ts\ldots
\ts+ \ts c_k=0$, \ for some \, $k\in \nn$, \. $c_i\in \zz[t]$,

\smallskip

\nin
\qquad {\it \textbf{Diagonal}}: \ \ \quad \.\ts $A(t) = \diag \ts P/Q$, \ for some \. $P,\ts Q\in \zz[x_1,\ldots,x_k]$, \. $k\ge 1$,

\smallskip

\nin
\qquad {\it \textbf{D-finite}}: \quad \quad \.\ts $c_0A \ts + \ts c_1A^\prime\ts + \ts\ldots\ts + \ts c_kA^{(k)}$,
\ for some \, $k\in \nn$, \. $c_i\in \zz[t]$,

\smallskip

\nin
\qquad {\it \textbf{D-algebraic}}:  \ \. $Q\bigl(t,A,A,\ldots,A^{(k)}\bigr)=0$, \ for some \, $k\in \nn$, \. $Q\in \zz[t,x_0,x_1,\ldots,x_k]$.

\medskip

\nin
It is well known and easy to see that
$$
 \text{\em Rational} \ \subsetneq \   \text{\em Algebraic} \ \subsetneq \  \text{\em Diagonal} \ \subsetneq \  \text{\em D-finite} \ \subsetneq \  \text{\em D-algebraic}
$$

\medskip

It is known that the cogrowth series \. $\Cog_\CS(t) \in$ {\em Rational} \. if and
only if $G$ is finite \cite{Kuk1}.
For example, for $G=\zz$ and $\CS=\{\pm 1\}$, we have:
$$\Cog_\CS(t) \, = \, \sum_{n=0}^{\infty} \. \tbinom{2n}{n} \ts t^{2n}
\,  = \, \diag \. \frac{1}{1-x-y} \,  = \,
\frac{1}{\sqrt{1-4t^2}} \ \in \,   \text{\em Algebraic}\..
$$
For $G=\zz^2$ and $\CS=\{(\pm 1, 0), \. (0,\pm1)\}$, the
cogrowth series \.
$\Cog_\CS(t)  = \sum_{n\ge 0} \tbinom{2n}{n}^2 \ts t^{2n}$  \. is
diagonal but not algebraic.\footnote{This was observed by Furstenberg \cite{Fur67}
via Schneider's theorem on transcendental numbers.  As noted in  \cite[p.~137]{Mel21},
this is also immediate from \ts
$\binom{2n}{n}^2\sim \frac{1}{\pi \ts n} \ts 16^n$.
Jungen's theorem can be used to show that the cogrowth series is non-algebraic
\emph{for all} \ts generating sets of~$\zz^2$.}
Diagonal GFs have coefficients which grow at most exponentially, so \ts
$\sum_{n\ge 0} n! \ts t^n$ \ts is D-finite but not a diagonal.
\defn{Christol's Conjecture} \ts claims that this is the only restriction:

\smallskip

\begin{conj}[Christol \cite{Chr90}] \label{conj:Christol}
Let \ts $A(t)=\sum_{n\ge 0} a_n t^n \in \zz[[t]]$.
Let \ts $|a_n| <c^n$ \ts for all $n\in \nn$ and some $c>0$,  and let \ts $A\in \text{D-finite}$.
Then \ts $A\in \text{Diagonal}$.
\end{conj}

\smallskip

Note that \emph{Euler's partition function}
$$P(t) \. := \. 1 \. + \. \sum_{n=1}^{\infty} p(n) \ts t^n \, = \,
\prod_{i=1}^{\infty} \. \frac{1}{1-t^i} \ \in \ \text{\em D-algebraic},
$$
see~\cite{MC44}. See also an explicit algebraic differential equation in \cite[$\S$2.5]{Pak18}.  Since \ts
$p(n)=e^{O(\sqrt{n})}$, it follows that $P(t)\notin$ {\em D-finite}.
In particular, Christol's Conjecture does not extend to D-algebraic GFs.

\medskip

\subsection{Proofs of Theorems~\ref{t:abelian} and~\ref{t:diag}} \label{ss:series-app}
We start with the following two results.

\begin{theorem}[{\rm Kuksov \cite[$\S$5.1]{Kuk2}}]
\label{t:Kuksov}
Let $G$ be a finitely generated abelian group with a finite symmetric generating set~$\ts\mathcal{S}$.
Then the cogrowth series \ts $\Cog_\CS(t) \in $ Diagonal.
\end{theorem}

For $G=\zz^d$, this result is folklore, see e.g.\ \cite[$\S$3.1.4]{Mis20}.
Note that Kuksov's formulation is different, but equivalent to ours.

\smallskip

\begin{thm}[{Adamczewski--Bell \cite[Thm.~9.1{\ts \small (i)}]{AB13}}]
\label{t:AB}
Let \ts $C(t)=\sum_{n\ge 0} c_n t^n \, \in $ Diagonal, let \ts $p$ \ts be a prime,
and let \ts $a\ge 1$, \ts $b \ge 0$ \ts be integers.  The following problem is \underline{decidable}\ts: \ts
\[
\exists  \ts n\in \nn \ : \
c_n \ts \equiv \ts b \ \mod p^a.
\]
\end{thm}

\smallskip

Theorems~\ref{t:abelian} and~\ref{t:diag} now follows easily by a combination
of these results and the Main Theorem~\ref{t:maintheorem}.

\smallskip

\begin{proof}[Proof of Theorem~\ref{t:abelian}]
Note that the proof of Theorem~\ref{t:Kuksov} in \cite[$\S$5.1]{Kuk2}
is completely constructive, giving \ts $\Cog_\CS = \diag \ts P_1/Q_1$ \ts and
\ts $\Cog_\CT=\diag \ts P_2/Q_2$ \ts for some explicit \ts
$P_1,P_2,Q_1,Q_2\in \zz[x_1,\ldots,x_2]$.
Let \ts $C(t)=\sum_{n\ge 0} c_n t^n :=\diag \ts \bigl(P_1/Q_1 - P_2/Q_2\bigr)$.
Apply Theorem~\ref{t:AB} to $C(t)$ with all possible $1\le b < p^a$, to check
if there is a solution for $b\not \equiv 0 \mod p^a$.  If not,
then we have \ts $c_n \ts \equiv \ts 0 \mod p^a$ \ts for all $n\in \nn$, as desired.
\end{proof}

\smallskip

\begin{proof}[Proof of Theorem~\ref{t:diag}]
Let \ts $p=2$,  \ts $a=40$, and let \ts $G=\UT(m,\zz)$ \ts be as in Theorem~\ref{t:maintheorem}.
Suppose every cogrowth series \ts $\Cog_\CS(t)$ \ts is a diagonal of polynomials
which are computable (given~$\CS$).  Then the same holds for the difference:
\ts $\Cog_\CS(t)-\Cog_\CT(t)= \diag P/Q$,  for every two symmetric generating
sets~$\CS$ and $\CT$ of~$G$, and some computable multivariate polynomials \ts $P,Q$.
By Theorem~\ref{t:AB}, the congruence
\[
\forall  \ts n\in \nn \ : \
\cog_{\mathcal{S}}(n) \ts \equiv \ts \cog_{\mathcal{T}}(n) \ \mod 2^{40}
\]
is decidable, a contradiction with Theorem~\ref{t:maintheorem}.
\end{proof}

\smallskip

\subsection{Non-D-algebraic cogrowth series}\label{ss:series-DAlg}
Ideally, one would want to give a construction of a non-D-algebraic cogrowth
series of a unitriangular group.  As an application of our tools we give
such a construction assuming there is a Diophantine equation with certain
properties.

\smallskip

Denote \ts $\bx=(x_1,\ldots,x_k)$, and let \ts $f \in \zz[x_1, \dots, x_k]$.
Consider a Diophantine equation \ts $f(\bx)=0$.  Denote by
\ts $\CR(f):=\{\bx\in \zz^k \.:\. f(\bx)=0\}$ \ts be the \emph{set of roots}.

\smallskip

We say that $f$ is \defn{sparse} \ts if all roots \ts $\bx \in \CR(f)$ \ts
have distinct $\ell^1$ norm: \ts $|\bx| \ne |\by|$ \ts for all \ts $\bx,\by \in \CR(f)$.
In this case we can assume that the roots of $f$ are ordered according to the norm: \ts
$\CR(f) = \{\bbr_1, \bbr_2,\ldots\}$, where \ts $|\bbr_1|< |\bbr_2|< \ldots$ \ts
For a sparse~$f$, we use \ts $\rho_i:=|\bbr_i|$.

Finally, for \ts $z\in \zz$, let \ts $\bin(z)$ \ts denote the number of $1$'s
in the binary expansion of~$|z|$.

\smallskip

\begin{conj} \label{conj:sparse}
There exists $k\in \nn$ and a sparse \ts $f\in \zz[x_1,\ldots,x_k]$ which satisfies:
\begin{itemize}
\item[$(1)$] \ $\rho_i$ \ts is even for all \ts $i\ge 1$,
\item[$(2)$] \ $\rho_{i +1}/\rho_i \to \infty$ \ts as \ts $i \to \infty$,
\item[$(3)$] \ for every integers \ts $a, b\ge 1$,  there exists \ts $i\ge 1$, s.t.\ \ts $\rho_i/2 \equiv a \mod 2^b$,
\item[$(4)$] \ for every integers $a, b, h \ge 1$, there exists some \ts $N=N(a,b,h) \ge 1$, s.t.\ for all \ts $i > N$ \ts
we have:
\[
\min \big\{\. y \, : \, \bin(c\ts\rho_i -  y) \leq a   \.\big\} \, \geq \, b\ts\rho_{i-1} \quad \text{ for all } \ \ 1 \leq c \leq h\ts.
\]
\end{itemize}
\end{conj}

\smallskip

\begin{theorem}\label{t:dalg}
Suppose Conjecture~\ref{conj:sparse} holds.  Then there exists an integer \ts $m\ge 1$ \ts
and a symmetric generating set \ts $\CS$ \ts of \. $\UT(m,\zee)$, s.t.\ the
cogrowth series \ts $\Cog_\CS(t)$ \ts is not D-algebraic.
\end{theorem}

\smallskip

We prove Theorem~\ref{t:dalg} in Section~\ref{s:proof-dalg}.  The proof is based
on the following result of independent interest.  It also explains the nature
of assumptions in the conjecture.

\smallskip

\begin{lemma}\label{l:dalglemma}
Let \ts $\{\lambda_n\} \in \nn^{\infty}$ \ts be an integer sequence
s.t.\ $\lambda_0 = 1$. Suppose there exists an increasing integer
sequence \ts $\{n_1 < n_2< \ldots\}$ \ts with the following properties:
\begin{itemize}
\item[$(1)$] \  $\lambda_{n_i}$ is odd for every $i\in \nn$,
\item[$(2)$] \   $n_{i + 1}/n_i \to \infty$ \. as \. $i \to \infty$,
\item[$(3)$] \ for every integers \ts $a, b\ge 1$,  there exists \ts $i\ge 1$, s.t.\ \ts $n_i \equiv a \mod 2^b$,
\item[$(4)$] \  for every $C, D\ge 1$, there exists \ts $N=N(C, D) > 0$, s.t.\ for every \ts $i_1, \dots, i_D > N$, if
\[
 n_{i_1} + \cdots + n_{i_D} - C \leq b_1 + \cdots + b_D \leq n_{i_1} + \cdots + n_{i_D}
\]
 for some nonnegative integers $b_1, \dots, b_D$, then either:
 \begin{itemize}
 \item \ $\lambda_{b_j}$ is even for at least one $j$.
 \item \ $\{b_1, \dots, b_D\}$ \ts and \ts $\{n_1, \dots, n_D\}$ \ts are equal up to rearrangement.
\end{itemize}
\end{itemize}
Then the sequence $\{\lambda_n\}$ is not D-algebraic.
\end{lemma}

\smallskip

For example, the sequence \ts $\{n_i = i! + i\}$ \ts satisfies properties~$(2)$ and~$(3)$ above.
Therefore, every integer sequence \ts $\{\lambda_n\}$,  where all \ts $\lambda_n$ \ts are odd
if and only if \ts $n = i! + i$ \ts for some~$i$, is not D-algebraic.

More generally, every integer sequence \ts $\{\lambda_n\}$, where $\lambda_n$ is odd whenever \ts
$n = i! + i$, and even when $n$ is not between $i! + i$ and $i! + 2i$ for some~$i$, is also not
D-algebraic. This is because we can take \ts $n_i := i! + i$ \ts  and property~$(4)$ will still hold.

\smallskip

\begin{remark}{\rm
If the sequence \ts $\{n_1, n_2, \dots\}$ \ts covers every index where \ts $a_n$ \ts is odd,
then condition $(4)$ follows from condition~$(3)$. This is because we could let $N$ be
large enough such that \ts $n_i > Dn_{i - 1}$ \t for all $i > N$. This case was
previously considered by Garrabrant and the first author.\footnote{Scott Garrabrant
and Igor Pak, unpublished notes (2015). }
}\end{remark}

\medskip

\section{Proof of Theorem \ref{t:maintheorem}}\label{s:proof-main}

The key idea in this proof will be to encode the existence of roots of an arbitrary Diophantine equation $f$ into statements about cogrowth in $\UT(m,\zz)$.
We proceed as follows. In Lemma~\ref{mainlemma} we show that words of a particular structure can compute the value of $f$ at integers. Then, in Lemmas \ref{l:wordsaresame} and \ref{subword} we extend our matrices so that this computation is true for a broader class of words.

Next, Lemmas \ref{l:modulos} and \ref{l:separation} allows us to turn the question of Theorem~\ref{t:maintheorem} into a statement about the existence of integer roots of an arbitrary Diophantine equation.  An explicit solution of Hilbert's 10th problem completes the proof.
%As the existence of such roots is known to be undecidable, this will complete the proof.

\smallskip

\subsection{Polynomials via matrix products} \label{ss:proof-main-first}
We start with the following key lemma.

\smallskip

\begin{lemma}\label{mainlemma}
Let \ts $f \in \zee[x_1, \dots, x_k]$ \ts and let \ts $\Deg:= \deg f$. Then there exists
matrices \. $P, Q, A_1, \dots, A_k \in \UT(m,\zee)$ \. for some \. $m \leq (\Deg +  1)\binom{\Deg + k}{k} + 2$,  such that
\[
PAQA^{-1}P^{-1}AQ^{-1}A^{-1} \, = \, I_m \. + \. f(x_1, \dots, x_k) \ts E_{1m}
\]
for all
\[
A = A_1^{x_1}A_2^{x_2} \cdots A_k^{x_k} \quad \ \text{and} \ \quad (x_1, \dots, x_k) \in \nn^k.
\]
\end{lemma}
\begin{proof}
Denote \ts $\bx = (x_1, \dots, x_k)$ and recall the multi-index notation~\eqref{multiindexbimonal}.
Write \ts  $f(\bx)$ \ts in the binomial basis \ts $\{\binom{\bx}{\bd}: \bd \in \enn^k\}$ \ts as follows:
\begin{equation} \label{basis}
 f(\bx) \, = \, \sum_{|\bd| \leq \Deg} \. b_\bd \ts \tbinom{\bx}{\bd} \quad \ \
 \text{for some \ \ \, $b_\bd \in \zz$, \ $\bd \in \nn^k$.}
\end{equation}

Let $p,q\ge 1$.
Denote by \ts $J_q$ \ts the $q\times q$ Jordan block with $1$'s on and above the diagonal. We have:
{\small
\begin{equation} \label{jordanpower}
J_q \, = \, \begin{bmatrix}
1 & 1 & 0 & \cdots & 0 & 0 \\
0 & 1 & 1 & \cdots & 0 & 0 \\
0 & 0 & 1 & \cdots & 0 & 0 \\
\vdots & \vdots & \vdots &\ddots & \vdots & \vdots \\
0 & 0 & 0 & \cdots & 1 & 1 \\
0 & 0 & 0 & \cdots & 0 & 1 \\
\end{bmatrix} \qquad \text{and} \qquad
\big(J_q\big)^p \, = \, \begin{bmatrix}
1 & \binom{p}{1} & \binom{p}{2} & \cdots & \binom{p}{q - 2} & \binom{p}{q - 1} \\
0 & 1 & \binom{p}{1} & \cdots & \binom{p}{q - 3} & \binom{p}{q - 2} \\
0 & 0 & 1 & \cdots & \binom{p}{q - 3} & \binom{p}{q - 4} \\
\vdots & \vdots & \vdots &\ddots & \vdots & \vdots \\
0 & 0 & 0 & \cdots & 1 & \binom{p}{1} \\
0 & 0 & 0 & \cdots & 0 & 1 \\
\end{bmatrix}.
\end{equation}
}

Now, for each \ts $\bd = (d_1, \dots, d_k)$ in the sum in~\eqref{basis},
define matrices \. $B_{\bd, i} \in \UT(|\bd| + 1,\zee)$ \.
as follows:
\begin{equation}\label{blocksdefined}
\left\{
\aligned
B_{\bd, 1} \, &: = \ J_{d_1 + 1} \oplus I_{d_2 \ts + \. \ldots \. + \ts d_k}   \\
B_{\bd, 2} \, &: = \  I_{d_1} \oplus J_{d_2 + 1} \oplus I_{d_3 \ts + \. \ldots \. + \ts d_k}  \\
&\ \vdots\\
B_{\bd, k} \, &: = \  I_{d_1 \ts + \. \ldots \. + \ts d_{k - 1}} \oplus J_{d_k + 1}
\endaligned\right.
\end{equation}

For example, if $\bd = (2, 3, 0, 1)$ then
{\small
\[
B_{\bd, 1} = \begin{bmatrix}
\red{1} & \red{1} & \red{0} & 0 & 0 & 0 & 0\\
\red{0} & \red{1} & \red{1} & 0 & 0 & 0 & 0\\
\red{0} & \red{0} & \red{1} & 0 & 0 & 0 & 0\\
0 & 0 & 0 & 1 & 0 & 0 & 0 \\
0 & 0 & 0 & 0 & 1 & 0 & 0\\
0 & 0 & 0 & 0 & 0 & 1 & 0\\
0 & 0 & 0 & 0 & 0 & 0 & 1\\
\end{bmatrix} \qquad B_{\bd, 2} = \begin{bmatrix}
1 & 0 & 0 & 0 & 0 & 0 & 0\\
0 & 1 & 0 & 0 & 0 & 0 & 0\\
0 & 0 & \red{1} & \red{1} & \red{0} & \red{0} & 0\\
0 & 0 & \red{0} & \red{1} & \red{1} & \red{0} & 0\\
0 & 0 & \red{0} & \red{0} & \red{1} & \red{1} & 0\\
0 & 0 & \red{0} & \red{0} & \red{0} & \red{1} & 0\\
0 & 0 & 0 & 0 & 0 & 0 & 1\\
\end{bmatrix}
\]
\[
B_{\bd, 3} = \begin{bmatrix}
1 & 0 & 0 & 0 & 0 & 0 & 0\\
0 & 1 & 0 & 0 & 0 & 0 & 0\\
0 & 0 & 1 & 0 & 0 & 0 & 0\\
0 & 0 & 0 & 1 & 0 & 0 & 0 \\
0 & 0 & 0 & 0 & 1 & 0 & 0\\
0 & 0 & 0 & 0 & 0 & \red{1} & 0\\
0 & 0 & 0 & 0 & 0 & 0 & 1\\
\end{bmatrix} \qquad B_{\bd, 4} =\begin{bmatrix}
1 & 0 & 0 & 0 & 0 & 0 & 0\\
0 & 1 & 0 & 0 & 0 & 0 & 0\\
0 & 0 & 1 & 0 & 0 & 0 & 0\\
0 & 0 & 0 & 1 & 0 & 0 & 0 \\
0 & 0 & 0 & 0 & 1 & 0 & 0\\
0 & 0 & 0 & 0 & 0 & \red{1} & \red{1}\\
0 & 0 & 0 & 0 & 0 & \red{0} & \red{1}\\
\end{bmatrix}
\]
}

Note that each of the $B_{\bd, i}$ contains one nontrivial Jordan block, highlighted in red above.  In the case where $d_i = 0$,
the Jordan block has size one.  The block is located between indices \ts $(d_1+ \ldots + d_{i - 1} + 1)$ \ts and \ts $(d_1+ \ldots  + d_{i} +1)$.  That means that the nontrivial block overlaps the nontrivial blocks of $B_{\bd, i - 1}$ and $B_{\bd, i + 1}$ in exactly one place.

Let $B_\bd = B_{\bd, 1}^{x_1}\cdots B_{\bd, k}^{x_k}$. Then the top-right entry of $B$ is given by
\begin{equation} \label{topright}
\big[B_\bd\big]_{1, |\bd| + 1} = \sum_{(j_1, \ts \ldots \ts , \ts  j_{k + 1}) \ : \ j_1=1, \. j_{k + 1} = |\bd| + 1} \. \big[B_{\bd, 1}^{x_1}\big]_{j_1, j_2} \. \big[B_{\bd, 2}^{x_2}\big]_{j_2, j_3} \. \cdots \. \big[B_{\bd, k}^{x_k}\big]_{j_k, j_{k + 1}}\,.
\end{equation}

We investigate which of the terms in the sum \eqref{topright} survive.
Since all the \ts $B_{\bd, i}$ \ts  are upper triangular we can only have a nonzero term if \.
$j_1 \leq j_2 \leq \cdots \leq j_{k + 1}$. By the block structure of the \ts $B_{\bd, i}$,
the only way to have  a nonzero term where \ts $j_{i} < j_{i + 1}$ \ts is if \ts $j_i$ \ts and \ts $j_{i + 1}$ \ts satisfy
\[
d_1 \ts + \ts \ldots \ts + \ts  d_{i - 1} + 1 \, \leq \, j_i \, < \, j_{i + 1} \, \leq \, d_1 \ts + \ts \ldots \ts + \ts  d_{i} + 1\ts.
\]

Therefore, there is only one nonzero term in the sum \eqref{topright},
given by \. $j_i = d_1 + \ldots + d_{i - 1} + 1$, for all~$i$.
This term is the product of the top-right entries of all the nontrivial
Jordan blocks in \ts $B_{\bd, 1}$ \ts to \ts $B_{\bd, k}$. By~\eqref{jordanpower}, this gives
\begin{equation}
\label{cornerbinom}
[B_\bd]_{1, |\bd| + 1} \, = \, \big[J_{d_1 + 1}^{x_1}\big]_{1, d_1 + 1} \. \cdots \. \big[J_{d_k + 1}^{x_k}\big]_{1, d_k + 1}  \, = \, \tbinom{x_1}{d_1 + 1 - 1} \. \cdots \. \tbinom{x_k}{d_k + 1 - 1}  \, = \, \tbinom{\bx}{\bd}.
\end{equation}

Now we need to arrange these parts to create $f$. For each $i$, define
\[
A_i \, := \, I_1 \. \oplus \. \Biggl[ \, \bigoplus_{|\bd|\le \Deg} B_{\bd, i} \. \Biggr] \. \oplus \. I_1\..
\]
Let $m$ be the size of \ts $A_i$\ts. For each \ts $|\bd|\le \Deg$,
let \ts $(\al_\bd, \be_\bd)$ \ts be the coordinates of the top-right entry
of the block in \ts $A_i$ \ts coming from \ts $B_{\bd, i}$. Then we can define
% \begin{align*}
$$
P \. := \. I_m \. + \. \sum_{|\bd|\le \Deg} \. E_{1, \al_\bd} \quad \text{and} \quad
Q \. := \. I_m \. + \. \sum_{|\bd|\le \Deg} \. b_\bd \ts E_{\be_\bd, m}\.,
% \end{align*}
$$
where the \ts $b_\bd$ \ts are the coefficients defined in~\eqref{basis}. The top-right corner of $PAQ$ is
\begin{align*}
[PAQ]_{1, m} \, = \, \sum_{1 \leq j_1, j_2 \leq m} \. [P]_{1j_1} \. [A]_{j_1j_2} \. [Q]_{j_2m}
\, = \,  \sum_{\bd_1, \bd_2} \. [A]_{\al_{\bd_1} \be_{\bd_2}} \. b_{\bd_2}\..
\end{align*}

But since the $A_i$'s were defined as block matrices, the only way for \ts
$[A]_{\al_{\bd_1}, \be_{\bd_2}}$ \ts to be nonzero is if \ts $\bd_1 = \bd_2$. Thus,
using \eqref{cornerbinom} this becomes
\begin{equation} \label{eq:cornerisright}
[PAQ]_{1, m} \ = \ \sum_{\bd} \, [A]_{\al_\bd, \be_\bd}\.  b_\bd  \ = \
\sum_\bd \, [B_\bd]_{1, |\bd| + 1} \. b_\bd  \ = \
\sum_\bd \, b_\bd \. \binom{\bx}{\bd}   \ = \ f(\bx)\ts.
\end{equation}

Now that we have a \ts $f(\bx)$ \ts in the top-right corner,
we need to make all the entries between this corner and the diagonal zero.
Let \ts $M = PAQA^{-1}$. Then we investigate its entries \ts $[M]_{ij}$.
Recall that
% \begin{equation}
$$
[M]_{ij} \ = \ \sum_{i \leq m_1 \leq m_2 \leq m_3 \leq j} \.
[P]_{i,m_1} \. [A]_{m_1,m_2} \. [Q]_{m_2,m_3} \. [A^{-1}]_{m_3,j} \label{matrixmult}
$$
%\end{equation}
%
and that the only above-diagonal nonzero entries of \. $P$ \ts are on the top row,
of \ts $Q$  \ts are in the right column, \ts and of \ts $A$ \ts
 are in neither the top row or right column.

 We  have the following cases:
\begin{itemize}
\item[$\circ$] \. If \. $i = j$, then \. $[M]_{i,j} = 1$ \. because \ts $M \in \UT(m,\zee)$.
\item[$\circ$] \. If \. $i > j$, then \. $[M]_{i,j} = 0$, analogously.
\item[$\circ$] \.  If \. $1 < i < j < m$, then we are above the diagonal of but not along the top or right edge of the matrix. Here the only terms in \eqref{matrixmult}, such that \. $[P]_{i,m_1} \neq 0$ \. will be those where \ts $m_1 = i$. Likewise we must have \ts $m_2 = m_3$, since \ts $m_3 < m$. Thus, we can ignore $P$ and $Q$ in the product, and conclude \. $[M]_{ij} = [AA^{-1}]_{ij} = 0$.
\item[$\circ$] \.  If \. $1 = i < j < m$, then we are on the top row of the matrix but not in the corner. Again we can ignore $Q$  because \ts $m_3 < m$. So \. $[M]_{i, j} = [PAA^{-1}]_{ij} = [P]_{ij}$.
\item[$\circ$] \.   If \. $1 = i < j = m$, then we are in the top-right corner of the matrix. Here $A^{-1}$ cannot contribute to the sum, since \ts $[A^{-1}]_{m_3, m}$ \ts is nonzero only when \ts $m_3 = m$. Thus, \. $[M]_{1,m} = [PAQ]_{1,m} = f(\bx)$ \. by~\eqref{eq:cornerisright}.
\end{itemize}

To summarize, $M$ is of the form
{\small
\begin{equation} \label{eq:emm}
M = \begin{bmatrix}
1 & [P]_{1, 2} & [P]_{1, 3} & [P]_{1, 4} & \cdots & [P]_{1, m - 1} & f(\bx) \\
0 & 1 & 0 & 0 & \cdots & 0 & \xii_1(\bx) \\
0 & 0 & 1 & 0 & \cdots & 0 & \xii_2(\bx) \\
0 & 0 & 0 & 1 & \cdots & 0 & \xii_3(\bx) \\
\vdots &\vdots & \vdots & \vdots  & \ddots & \vdots & \vdots \\
0 & 0 & 0 & 0& \cdots & 0 & 1
\end{bmatrix}
\end{equation}
}
where the \ts $\xii_i(\bx)$ \ts denote some polynomials.

\smallskip

Note that $P$ is nonzero only in the first row and zero in the top-right corner.
Thus, the same holds for $P^{-1}$. Therefore, we can right-multiply \eqref{eq:emm}
by $P^{-1}$ to get
{\small
\begin{equation}\label{mpinv}
MP^{-1} \, = \, \begin{bmatrix}
1 & 0 & 0 & 0& \cdots  & f(\bx) \\
0 & 1 & 0 & 0 & \cdots & \xii_1(\bx) \\
0 & 0 & 1 & 0 & \cdots & \xii_2(\bx) \\
0 & 0 & 0 & 1 & \cdots & \xii_3(\bx) \\
\vdots&\vdots & \vdots & \vdots  & \ddots  & \vdots \\
0 & 0 & 0 & 0& \cdots  & 1
\end{bmatrix}.
\end{equation}
}

Similarly, \ts $P^{-1}M$ \ts must be equal to $M$ except possibly in the first row. But \ts
$P^{-1}M = AQA^{-1}$ \ts is the product of three matrices whose first rows are trivial.
Thus, \ts $P^{-1}M$ \ts must also be trivial in the first row. We conclude:
{\small
\begin{equation} \label{pinvm}
P^{-1}M \, = \, \begin{bmatrix}
1 & 0 & 0 & 0& \cdots & 0 \\
0 & 1 & 0 & 0 & \cdots & \xii_1(\bx) \\
0 & 0 & 1 & 0 & \cdots & \xii_2(\bx) \\
0 & 0 & 0 & 1 & \cdots & \xii_3(\bx) \\
\vdots&\vdots & \vdots & \vdots  & \ddots  & \vdots \\
0 & 0 & 0 & 0& \cdots  & 1
\end{bmatrix}.
\end{equation}
}
Combining \eqref{mpinv} and \eqref{pinvm}, we get
\begin{align*}
& PAQA^{-1}P^{-1}AQ^{-1}A^{-1} \ = \ \left(PAQA^{-1}P^{-1} \right) \left(AQ^{-1}A^{-1} \right)^{-1} \\
& \hskip1.6cm = \ MP^{-1} \left(P^{-1}M \right)^{-1} \ = \ I_m + f(\bx),
\end{align*}
as desired.

We now consider the size of $m$. There are exactly \. $\binom{\Deg + k}{k}$ \.
possible multi-indices \ts $\bd$ \ts with \. $|\bd| \leq \Deg$.
Each of these contributes at most \ts $(\Deg + 1)$ \ts to  the size of~$A_i$\ts,
and we get an additional $1$ from each~$I_1$.
This gives \. $m \leq (\Deg +  1)\binom{\Deg + k}{k} + 2$.
\end{proof}

\begin{corollary}\label{rootsarecogrowth}
A word of the form
\[
P W_1 Q W_2P^{-1} W_3 Q^{-1} W_4 \qquad \text{where} \qquad W_1 = W_2^{-1} = W_3 = W_4^{-1} = A_1^{x_1}\cdots A_k^{x_k}
\]
is a cogrowth word \. if and only if \.  $\bx = (x_1, \dots, x_k)$ is a root of $f$.

\end{corollary}

\smallskip

\subsection{Larger families of words}\label{ss:proof-main-second}
We now have the tools to evaluate Diophantine equations, but in order to be able to eliminate extraneous words, we will need to extend the matrices defined in Lemma~\ref{mainlemma} to new matrices. Therefore the next lemmas will reduce the problem to Corollary \ref{rootsarecogrowth}. Note that we will continue referring to the new matrices as $A_i$, $P$, and $Q$ in order to connect their roles to those in Lemma~\ref{mainlemma}.

First, we extend our matrices so that the four words $W_1$, $W_2$, $W_3$, $W_4$ do in fact need to be inverses as in the statement of Lemma~\ref{mainlemma}.

\begin{lemma}\label{l:wordsaresame}
Suppose \ts $f \in \zee[x_1, \dots, x_k]$ \ts has degree \ts $\Deg:=\deg f$.
Then there exists  matrices \. $P, Q$, $A_1, \dots, A_k \in \UT(m,\zee)$ \. for some \.
$m \leq 4(\Deg + 1)\binom{\Deg + k}{k} + 8$, such that the conclusion of
Corollary \ref{rootsarecogrowth} holds, and such that every word of the form
\[
P W_1 Q W_2P^{-1} W_3 Q^{-1} W_4 \qquad \text{where \ \ $W_i \in \langle A_1^{\pm1},  \dots, A_k^{\pm k}\rangle$}
\]
is a cogrowth word  \. only if \. $W_1 = W_2^{-1} = W_3 = W_4^{-1}$.
\end{lemma}

\begin{proof}
Let $P', Q', A_1', \dots, A_k'$ be the matrices produced by Lemma~\ref{mainlemma}. Define
{\small
\begin{align*}
P \, := \, \begin{bmatrix}
P' & 0 & 0 & 0 \\
0 & I_m & 0 & I_m \\
0 & 0 & I_m & 0 \\
0 & 0 & 0 & I_m
\end{bmatrix}\., \qquad
Q \, := \, \begin{bmatrix}
Q' & 0 & 0 & 0 \\
0 & I_m & 0 & 0 \\
0 & 0 & I_m & I_m \\
0 & 0 & 0 & I_m
\end{bmatrix}\., \qquad
A_i \, := \, \begin{bmatrix}
A_i' & 0 & 0 & 0 \\
0 & I_m & 0 & 0 \\
0 & 0 & I_m & 0 \\
0 & 0 & 0 & A_i'
\end{bmatrix} \..
\end{align*}
}
If $W_1 = A_{i_1}^{\pm 1} \cdots A_{i_s}^{\pm 1}$, then define
\[
W_1' \, := \, (A_{i_1}')^{\pm 1} \cdots (A_{i_s}')^{\pm 1}
\]
and analogously for $W_2', W_3', W_4'$.
A computation then shows
{\small
\[
P W_1 Q W_2P^{-1} W_3 Q^{-1} W_4
 = \begin{bmatrix}
V & 0 & 0 & 0 \\
0 & I_m & 0 &  W_3'W_4'(I_m - W_2'W_1') \\
0 & 0 & I_m & W_4'(I_m - W_2'W_3') \\
0 & 0 & 0 & W_1'W_2'W_3'W_4'
\end{bmatrix}
\]
}
where \. $V = P' W_1' Q' W_2'(P')^{-1} W_3' (Q')^{-1} W_4'$.
The construction in Lemma~\ref{mainlemma} shows that Corollary \ref{rootsarecogrowth} holds.

Moreover, for this matrix to be the identity, we must have
$$
W_3'W_4'(I_m - W_2'W_1') \. = \.  W_4'(I_m - W_2'W_3') \. = \. 0 \quad \text{and} \quad
W_1'W_2'W_3'W_4' \. = \. I_m\,,
$$
which implies \ts $W_1' = (W_2')^{-1} = W_3' = (W_4')^{-1}$. This gives  \ts $W_1 = W_2^{-1} = W_3 = W_4^{-1}$ as required.
\end{proof}

We now know that the $W_i$ need to evaluate to the same matrix, but Lemma~\ref{mainlemma} is only able to speak about subwords. So we must extend our matrices again, this time so that the only possible cogrowth words are equivalent to subwords.

We do this by noticing that if we flip the Jordan block construction from Lemma~\ref{mainlemma} so the blocks go from bottom-right to top-left instead, then instead of evaluating monomials the above-Jordan-block terms will  be zero. That allows us to prove the following:

\begin{lemma}\label{subword}
Let \ts $f \in \zee[x_1, \dots, x_k]$ \ts with \ts  $\Deg = \deg f \geq 2$. Then there exists
matrices \. $P, Q, A_1, \dots, A_k \in \UT(m,\zee)$ \. for some
\[
m \, \leq \, 4(\Deg + 1)\tbinom{\Deg + k}{k} \. + \. 8 \. + \. \tfrac{1}{2}\tbinom{\Deg + k}{k}(\Deg + 1)^3\ts,
\]
 such that the conclusion of Corollary~\ref{rootsarecogrowth} holds, and such that every word of the form
\begin{equation} \label{niceword}
P W_1 Q W_2P^{-1} W_3 Q^{-1} W_4
\end{equation}
where \ts $W_i \in \langle A_1^{\pm1},  \dots, A_k^{\pm k}\rangle $, is a cogrowth word \ts
only if \ts $W_1 = W_2^{-1} = W_3 = W_4^{-1} = A_1^{x_1} \cdots A_k^{x_k}$ \ts for some integers \ts
$x_1, \dots, x_k$.

\end{lemma}

\begin{proof}
Let $P', Q', A_1', \dots, A_k'$ be the matrices produced by Lemma~\ref{mainlemma}. We consider the structure of matrices in  $\langle (A_1')^{\pm 1}, \dots, (A_k')^{\pm} \rangle$ more deeply. Each consists of a collection of blocks defined as $B_{\bd, i}$ in \eqref{topright}. Fix any particular~$B_\bd$. By construction, it is of size \ts $|\bd| + 1$.

For any matrix \ts $X \in \UT(L,\zee)$, let \ts $\varphi(X)$ \ts be the matrix obtained by reflecting $X$ along the main antidiagonal. Then $\Phi: X  \mapsto \varphi(X)^{-1}$ is an automorphism of $\UT(L,\zee)$. Now, $B_{\bd, 1}, \dots, B_{\bd, k}$ have their nontrivial blocks arranged from top left to bottom right; so $\Phi(B_{\bd, 1}), \dots, \Phi(B_{\bd, k})$ have their nontrivial blocks arranged from bottom right to top left.

For example, if
\[
B_{\bd, 1} = \begin{bmatrix}
\red{1} & \red{1} & \red{0} & 0 & 0 & 0\\
\red{0} & \red{1} & \red{1} & 0 & 0 & 0\\
\red{0} & \red{0} & \red{1} & 0 & 0 & 0 \\
0 & 0 & 0 & 1 & 0 & 0\\
0 & 0 & 0 & 0 & 1 & 0\\
0 & 0 & 0 & 0 & 0 & 1 \\
\end{bmatrix} \qquad B_{\bd, 2} = \begin{bmatrix}
1 & 0 & 0 & 0 & 0 & 0\\
0 & 1 & 0 & 0 & 0 & 0\\
0 & 0 & \blu{1} & \blu{1} & \blu{0} & \blu{0} \\
0 & 0 & \blu{0} & \blu{1} & \blu{1} & \blu{0}\\
0 & 0 & \blu{0} & \blu{0} & \blu{1} & \blu{1} \\
0 & 0 & \blu{0} & \blu{0} & \blu{0} & \blu{1} \\
\end{bmatrix}
\]

then

\[
\Phi(B_{\bd, 1}) = \begin{bmatrix}
1 & 0 & 0 & 0 & 0 & 0\\
0 & 1 & 0 & 0 & 0 & 0\\
0 & 0 & 1 & 0 & 0 & 0 \\
0 & 0 & 0 & \red{1} & \red{-1} & \red{1}\\
0 & 0 & 0 & \red{0} & \red{1} & \red{-1}\\
0 & 0 & 0 & \red{0} & \red{0} & \red{1} \\
\end{bmatrix} \qquad \Phi(B_{\bd, 2}) = \begin{bmatrix}
\blu{1} & \blu{-1} & \blu{1} & \blu{-1} & 0 & 0\\
\blu{0} & \blu{1} & \blu{-1} & \blu{1} & 0 & 0\\
\blu{0} & \blu{0} & \blu{1} & \blu{-1} & 0 & 0 \\
\blu{0} & \blu{0} & \blu{0} & \blu{1} & 0 & 0\\
0 & 0 & 0 & 0 & 1 & 0\\
0 & 0 & 0 & 0 & 0 & 1 \\
\end{bmatrix}
\]

\begin{sublemma} \label{automorphism}
A matrix $W \in \langle B_{\bd, 1}^{\pm1},  \dots, B_{\bd, k}^{\pm 1}\rangle$ is equal to $B_{\bd, 1}^{x_1} \cdots B_{\bd, k}^{x_k}$ for some integers $x_1, \dots, x_k$ if and only if $\Phi(W)$ is zero outside of the nontrivial Jordan blocks of $\Phi(B_{\bd, 1}), \dots, \Phi(B_{\bd, k})$.
\end{sublemma}
\begin{proof}
The forward direction is immediate: because the nontrivial Jordan blocks of the $\Phi(B_{\bd, i})$ are in bottom right to top left order,
the matrix
\[
\Phi \left(B_{\bd, 1}^{x_1} \cdots B_{\bd, k}^{x_k} \right) \, = \, \Phi(B_{\bd, 1})^{x_1} \. \cdots \. \Phi(B_{\bd, k})^{x_k}
\]
will not have any nonzero entries  outside the nontrivial Jordan blocks of the matrices \ts $\Phi(B_{\bd, i})$.

Conversely, suppose $\Phi(W)$ is zero outside of the nontrivial Jordan blocks of \ts $\Phi(B_{\bd, i})$.
Since~$W$ is in the subgroup generated by the \ts $B_{\bd, i}$, we can write
\begin{equation} \label{spanning}
W \ts = \ts B_{\bd, j_1}^{\varepsilon_1} \cdots B_{\bd, j_m}^{\varepsilon_m}
\end{equation}
for some integer~$m$, indices \ts $1 \leq j_m \leq k$, and exponents \ts
$\varepsilon_m = \pm 1$. Let \ts $y_1, \dots, y_k$ \ts be the net number of \ts
$B_{\bd, 1}, \dots, B_{\bd, k}$ \ts in expression \eqref{spanning}. In other words, we have:
\[
y_i \. = \. \sum_{s\ :\ j_s = i} \varepsilon_s\..
\]
By assumption, $\Phi(W)$ agrees with \ts $\Phi(B_{\bd, 1})^{y_1} \cdots \Phi(B_{\bd, k})^{y_k}$ \ts
outside of the nontrivial Jordan blocks. Fix some index $\alpha, \beta$ within the nontrivial Jordan block of $B_{\bd, \gamma}$. Then (\ref{spanning}) implies that
\[
\Phi(W) \. = \. \Phi(B_{\bd, j_1})^{\varepsilon_1} \ \cdots \. \Phi(B_{\bd, j_m})^{\varepsilon_m}\..
\]
Note that the only terms that can contribute to the $\alpha, \beta$ index are those where \ts $j_s = \gamma$. This means
\[
[\Phi(W)]_{\alpha, \beta} = [\Phi(B_{\bd, \gamma})^{y_\gamma}]_{\alpha, \beta} = [\Phi(B_{\bd, 1})^{y_1} \cdots \Phi(B_{\bd, k})^{y_k}]_{\alpha, \beta}
\]
Since this holds for any $\alpha, \beta$ we get
\[
\Phi(W) = \Phi(B_{\bd, 1})^{y_1} \cdots \Phi(B_{\bd, k})^{y_k} = \Phi(B_{\bd, 1}^{y_1} \cdots B_{\bd, k}^{y_k})\..
\]
The result follows since $\Phi$ is a bijection.
\end{proof}

The next sublemma will allow us to force particular entries in $\Phi(W)$ to be zero.
\begin{sublemma}\label{pickout}
Let \ts $V\in \UT(q,\zee)$ \ts and let \ts $1 < a \leq b <  q$. Then
%\begin{multline*}
$$
\big(I_q + E_{1, a}\big) V \big(I_q + E_{b, L}\big) V^{-1} \big(I_q + E_{1, a}\big)^{-1} V
\big(I_q + E_{b, q}\big)^{-1} V^{-1}  \, = \, I_q \. + \. [V]_{a,b}\ts E_{1, q}\..
$$
%\end{multline*}
\end{sublemma}

\begin{proof}
The left-hand side is equal to
\[
(I_q + E_{1, a}) V (I_q + E_{b, q}) V^{-1} (I_q - E_{1, a}) V (I_q - E_{b, q}) V^{-1}
\]
Expanding this and using the fact that $V$ and $V^{-1}$ are upper triangular gives \.
$I_q + E_{1, a} V E_{b, q} V^{-1}$. This equals the right-hand side.
\end{proof}

\smallskip

To finish the proof of Lemma~\ref{l:wordsaresame}, we
construct our matrices as follows. Let \. $P'', Q'', A_1'', \dots, A_k''$ \.
be the matrices obtained in Lemma~\ref{l:wordsaresame}. For every $B_{\bd}$
in the construction of $A_i'$, and every $(\alpha, \beta)$ above the
nontrivial Jordan blocks of \ts $\Phi(B_{\bd, i})$,
%append (in the sense of $\oplus$) a copy of
let
\begin{equation*}
\aligned
P \, & := \, P'' \. \oplus \. \left( I_{|\bd| + 3} + E_{1, \alpha + 1} \right) \\ % & \text{ to } P'' \\
Q \, & := \, Q'' \. \oplus \. \left(  I_{|\bd| + 3} + E_{\beta + 1, |\bd| + 3} \right) \\ %& \text{ to } Q'' \\
A_i \, & := \, A_i'' \. \oplus \.  I_1 \oplus \Phi(B_{\bd, i}) \oplus I_1 % & \text{ to } A_i''
\endaligned
\end{equation*}
for all \ts $1\le i \le k$.
%
%Call the resulting matrices $P$, $Q$, and $A_1, \dots, A_k$.
%
There are at most $\binom{\Deg + k}{k}$ of the $B_\bd$'s, and for each of them we append at most $\frac{1}{2}(\Deg + 1)^2$ new matrices of size at most $\Deg + 1$. Therefore these new matrices have size
\[
m \, \le \,
4(\Deg + 1)\binom{\Deg + k}{k} + 8 + \frac{1}{2}\binom{\Deg + k}{k}(\Deg + 1)^3,
\]
as desired.

Suppose a word of the form (\ref{niceword}) is cogrowth. Then by Lemma~\ref{l:wordsaresame} we have $W_1 = W_2^{-1} = W_3 = W_4^{-1}$. Therefore, by construction and Sublemma \ref{pickout} all of the entries of $\Phi(W_1)$ outside of the nontrivial Jordan blocks are zero. Then, Sublemma \ref{automorphism} implies that $W_1 =
A_1^{y_1} \cdots A_k^{y_k}$ for the $y_i$ defined in Sublemma \ref{automorphism}.  This completes the proof of Lemma~\ref{l:wordsaresame}.
\end{proof}

\smallskip

\begin{corollary} \label{uniqueshortest}

For a fixed root $\bx = (x_1, \dots, x_k)$ of $f$, the word
\[
V = A_1^{x_1} \circ \cdots \circ A_k^{x_k}
\]
is the unique shortest word that evaluates to $A_1^{x_1}\cdots  A_k^{x_k}$.
\end{corollary}

\begin{proof}
We only need to prove the case $k \geq 2$.  Suppose to the contrary, there is some other word $V'$ which also evaluates to \.
$A_1^{x_1} \cdots A_k^{x_k}$. Since the net number of $A_i$'s in $V'$ needs to be $x_i$, it must be that $V'$
is some nontrivial permutation of $V$.

This means there exists some \ts $j_1 < j_2$\ts, such that an $A_{j_2}^{\pm 1}$ appears before an $A_{j_1}^{\pm 1}$ in the word~$W_i$.
But then the above-diagonal entry in the block corresponding to \ts
$\binom{j_1}{1}\binom{j_2}{1}$ \ts will be nonzero, so this cannot be a cogrowth word.
\end{proof}

\smallskip

\subsection{The construction} \label{ss:main-proof-three}
We are now ready to construct our generating sets \ts $\mathcal{S}$ \ts and \ts $\mathcal{T}$ as in Theorem~\ref{t:maintheorem}. For a fixed polynomial \ts $f\in \zz[x]$, let \. $P'$, $Q'$, $A_1'$, \dots, $A_k' \in \UT(m,\zee)$ \. be the matrices given by Lemma~\ref{subword}.
Construct new matrices \. $A_i := A_i' \oplus I_3$, for \ts $1 \leq i \leq k$, and let
\begin{align*}
P \. := \. P' \oplus \begin{bmatrix}
1 & 1 & 0 \\
0 & 1 & 0 \\
0 & 0 & 1
\end{bmatrix}\,, \quad
Q \. := \. Q' \oplus \begin{bmatrix}
1 & 0 & 0 \\
0 & 1 & 1 \\
0 & 0 & 1
\end{bmatrix}\,, \quad
R \. := \. I_m \oplus \begin{bmatrix}
1 & 0 & -1 \\
0 & 1 & 0 \\
0 & 0 & 1
\end{bmatrix}.
\end{align*}

Denote by \ts $\ce_m=\{I_m\pm E_{i,i+1} \.:\. 1\le i <m\}$ \ts the standard generating
set of \ts $\UT(m,\zz)$. Fix be a positive integer \ts $u$ \ts to be determined later.
Let
\begin{equation}\label{eq:ST}
\aligned
\CS \, &:= \, \{A_1^{\pm 1}, A_2^{\pm 1}, \dots, A_k^{\pm 1}\} \ \cup \ u \cdot \left\{P^{\pm 1}, Q^{\pm 1} \right\} \ \cup \ u^{10} \cdot \ce_{m+3}\., \ \text{and} \\
\CT \, &:= \, \CS \ \cup  \ u^5\cdot \left\{R^{\pm 1} \right\},
\endaligned
\end{equation}
where by \ts $n\cdot X$ \ts we denote $n$ copies of the set~$X$.

Our next lemma will exploit the modular condition in Theorem \ref{t:maintheorem} to eliminate any word that does not fit the pattern of Lemma~\ref{subword}.

\begin{lemma}\label{l:modulos}
Let $f \in \zee[x_1, \dots, x_k]$, and define $\mathcal{S}$ and $\mathcal{T}$ as in~\eqref{eq:ST}.
Let $c_n$ be the number of cogrowth words of length $n$ of the form
\begin{equation*}
P V_1 Q V_2P^{-1} V_3 Q^{-1} V_4\,, \quad \text{where \ $V_i$ \ are words in \ $\langle A_1^{\pm1},  \dots, A_k^{\pm k}\rangle $.}
\end{equation*}
 Then:
\[
\cog_\mathcal{T}(n) \. - \. \cog_\mathcal{S}(n) \, \equiv \, 2\ts n \ts (n - 1) \ts c_{n - 1} \ts u^9 \ \mod  u^{10}.
\]
\end{lemma}

\begin{proof}
First, note that we can ignore all words that contain any of the standard generators.
By construction, such words will appear a multiple of $u^{10}$ times.

Second, note that the left-hand side counts the number of cogrowth words
that are in \ts $\langle \mathcal{T} \rangle$ \ts but not in~$\ts\langle \mathcal{S} \rangle$.
This corresponds to words with at least one \ts $R^{\pm 1}$.
However, words with two or more $R^{\pm 1}$ will be eliminated by the modulo condition.

Next, there is a bijection between words containing one $R$ and those containing one $R^{-1}$ given by reversing the order of the word and inverting all the elements. So let us look only at words that contain just an~$R$. This gives a factor of~$2$ on the right hand side.

In order to cancel out the $-1$ in $R$ we can only use copies of \ts $P^{\pm 1}$ \ts and \ts $Q^{\pm 1}$. But every word with an~$R$ and at least five of these will also be eliminated since the total weight would be divisible by $u^{10}$. So the only possible words that remain have some cyclic permutation of $PQP^{-1}Q^{-1}$, which gives the factor of $u^9$.

Because any cyclic permutation of a cogrowth word is still cogrowth, we can take the unique word that starts with $P$. This gives a factor of $n$ on the right hand side.

Finally, note that $R$ commutes with $P, Q$, and all the $A_i$. Since our word has exactly one $R$, we can just ignore it in counting words by looking at words of length \ts $(n - 1)$. This gives us one more factor of \ts $(n  - 1)$ on the right-hand side.
The result counts exactly \ts $c_{n - 1}$.
\end{proof}

\smallskip

The following two corollaries relate this lemma to whether or not the polynomial~$f$ has integer roots.

\smallskip

\begin{corollary}\label{onedirection}
Let \ts $f\in \zz[x_1,\ldots,x_k]$ \ts be a polynomial with no integer roots,  Then
$$
\cog_\mathcal{T}(n) \. - \. \cog_\mathcal{S}(n) \, \equiv \, 0 \ \mod u^{10}.
$$
\end{corollary}

\smallskip

In a different direction, we have:

\begin{corollary}\label{c:otherdirection}
Let \ts $f\in \zz[x_1,\ldots,x_k]$ \ts be a polynomial with an integer root \ts $\bx\in \zz^k$.
Suppose that \ts $|\bx|$ \ts is even, and \ts $|\bx|$ \ts is minimal among all integer roots
of~$f$.  Let \ts $u = 16$ \ts and let \ts $\CS,\CT$ \ts be defined by~\eqref{eq:ST}.
Then:
$$\cog_\mathcal{T}(4|\bx| + 5) \. - \. \cog_\mathcal{S}(4|\bx| + 5) \ \not\equiv \ 0 \ \mod u^{10}.
$$
\end{corollary}

\begin{proof}
By Lemma~\ref{l:modulos}, we have:
%\begin{multline*}
$$
\cog_\mathcal{T}\big(4\ts |\bx| + 5\big) \. - \.  \cog_\mathcal{S}\big(4\ts |\bx| + 5\big) \ \equiv \ 2\big(4\ts |\bx|  + 5\big)\big(4\ts |\bx|  + 4\big) \ts c_{4\ts |\bx|  + 4} \ts  16^9 \ \mod \ts 16^{10}.
$$
%\end{multline*}
Since $|\bx|$ is minimal, the only way to have a cogrowth word in $c_{4|\bx|  + 4}$ is to let $V_i = A_1^{x_1} \circ \cdots \circ A_k^{x_k}$ by Lemma~\ref{subword} and Corollary \ref{uniqueshortest}. So $c_{4|\bx|  + 4} = 1$. Because $|\bx|$ is even, the right hand side has only at most $1 + 0 + 2 + 36 = 39$ factors of $2$. That means that it not not zero modulo $16^{10}$, as desired.
\end{proof}

\smallskip

\begin{rem}\label{r:separation} {\rm
Unfortunately, not every polynomial has a root satisfying the conditions of Corollary~\ref{c:otherdirection}. For example, the polynomial \ts $f(x_1, x_2) = x_1^2 - 13x_2^2 - 1$ \ts has four solutions with minimal $\ell^1$-norm, namely \ts $(\pm 649, \pm 180)$. This would imply that $c_{3317} = 4$, introducing an extra factor of $2$ to the right-hand side and making the two sides congruent.}
\end{rem}

\smallskip

To avoid the issue in the remark above, we introduce an auxiliary variable which will
separate out the $\ell^1$ norms of all integer roots.

\smallskip

\begin{lemma}  \label{l:separation}
There exists a map \. $\Phi: \zee[x_1, \dots, x_k] \ts \to \ts \zee[y_1, \dots, y_{k + 1}]$,
such that for all \. $\widetilde{g}=\Phi(g)$ \. we have:
\begin{itemize}
\item[$\circ$] \ polynomials \ts $g$ \ts  and \ts $\widetilde{g}$ \ts have the same (possibly infinite) number of integer roots, \eqnum \label{rootscorrespond}
\item[$\circ$] \ $\bx\in \zz^{k+1}$ \. is an integer root of \. $\widetilde{g}$ \ \. $\Rightarrow$ \ \. $|\bx|$ \. is even,  \eqnum \label{evennorms}
\item[$\circ$] \ $\bx, \ts \by\in \zz^{k+1}$ \. are integer roots of \. $\widetilde{g}$ \ \. $\Rightarrow$ \ \. $|\bx|\ne |\by|$, \eqnum \label{uniqueroots}
\item[$\circ$] \ $\deg \ts \widetilde{g} \. \leq \. \max\{2 \deg g, \ts 4k + 12\}$. \eqnum\label{degreebound}
\end{itemize}
\end{lemma}

\begin{proof}
Let \. $v=v(\by) \ts := \ts 4(y_1^2 + y_2^2 + \cdots + y_k^2 + 1)$, and let
\[
\widetilde{g}(y_1, \dots, y_{k + 1}) \. = \. \Phi(g) \, := \, g(y_1, \dots, y_k)^2 \. + \.
\left( -y_{k + 1} \. + \. v^{k + 3} \. + \. \sum_{i = 1}^{k} \. y_i v^{i + 1} \. + \. \sum_{i = 1}^{k} \. y_i \right)^2.
\]
Note that condition \eqref{degreebound} is clearly satisfied.

In order for $\widetilde{g}$ to have a root, we must have \. $g(y_1, \dots, y_k )   = 0$ \. and
\[
y_{k + 1} \, = \, v^{k + 3} \. + \. \sum_{i = 1}^k \. y_i v^{i + 1} \. + \. \sum_{i = 1}^k \. y_i\..
\]
This implies \eqref{rootscorrespond}.

Next, suppose \ts $\bbr = \{y_1, \dots, y_{k + 1}\}$ \ts is an integer root of \ts $\widetilde{g}$.
Because $v$ is even, we have:
\[
|\bbr| \ \equiv \ |y_1| \. + \. \ldots \. + \. |y_{k}| \. + \. 0 \. + \.  y_1 \. + \.  \ldots \. + \.
y_k \ \equiv \ 0 \ \. \mod 2,
\]
which proves \eqref{evennorms}.

On the other hand, observe that
\begin{align*}
\left| |\bbr| \. - \. v^{k + 3}\right| \ &\leq \
\sum_{i = 1}^k \ts |y_i| \. + \. \left| \sum_{i = 1}^{k} \. y_i v^{i + 1} \. + \.
\sum_{i = 1}^{k} \. y_i \right| \ \leq \ \sum_{i = 1}^k |y_i| \left( 2 + v^{i + 1} \right) \\
& \leq \ (2 + v^{k + 1}) \. \sum_{i = 1}^k \. |y_i| \ \leq \  (2 + v^{k + 1}) \frac{v}{4} \ \leq \ v^{k + 3} - (v - 1)^{k + 3} \..
\end{align*}
This implies that if \. $\wt g(\bx)=\wt g(\by)$, then \. $v(\bx)=v(\by)$.

Now suppose that \ts $\bx,\by\in \zz^{k+1}$ \ts are roots of~$\ts \wt g$ \ts such that \. $|\bx|=|\by|$.
From above, \ts $v(\bx)=v(\by)$.  Write \ts $Y:=|\by|-v(\by)^{k+3}$ \ts as a polynomial in \ts $y_1,\ldots,y_k$ \ts
and observe that \ts $y_i$'s are uniquely determined by the integrality.  For example, \ts $y_1$ \ts is the closest
integer to \ts $Y/v^{k + 1}$, etc.  The same argument for \ts $\bx$ \ts shows
that \ts $\bx=\by$, which implies~\eqref{uniqueroots}.  This finishes the proof of the lemma.
\end{proof}

\smallskip

We can now complete the proof of Theorem \ref{t:maintheorem}. Suppose an algorithm exists that determines whether or not, for arbitrary generating sets $\mathcal{S}$ and $\mathcal{T}$, we have
\begin{equation}\label{cogcond}
\exists \ts n \geq 0 \ : \ \cog_{\mathcal{S}}(n) \, \not\equiv \, \cog_{\mathcal{T}}(n) \mod p^a.
\end{equation}
Then we could use this algorithm to determine whether or not a Diophantine equation $g(x_1, \dots, x_k)$ has an integer root as follows. First construct $\widetilde{g}$ as in Lemma~\ref{l:separation}. Then construct $\mathcal{S}$ and $\mathcal{T}$ with $f = \widetilde{g}$ and $u = 16$ as in Lemma~\ref{subword}. By Corollaries~\ref{onedirection} and~\ref{c:otherdirection}, polynomial $\widetilde{g}$, and thus $f$, has a root if and only if \eqref{cogcond} holds with \ts $p=2$ and $a=40$, so \ts $p^a = u^{10}$.

Finally, Jones \cite{Jon82} shows that Diophantine problems over $\enn$ are undecidable for polynomials of degree at most $96$ in $21$ variables. By a standard reduction (see e.g.~\cite[Thm~3.3]{Gas21}), the Diophantine problem over $\zee$ is undecidable for \ts $\deg g = 192$ \ts and \ts $k = 63$. Then \ts $\Deg = \deg \widetilde{g}= 384$, which by Lemma~\ref{subword} gives the desired bound \ts $m \leq 9.6 \cdot 10^{85}$.  This completes the proof of Theorem \ref{t:maintheorem}. \qed

\smallskip

\begin{remark}\label{r:Jones} {\rm
In fact, Jones~\cite{Jon82} (see also~\cite{Gas21}),
gives several pairs (degree, number of variables) which give
rise to a minimal Diophantine equation.
Of these, we chose the one which gives the smallest bound on~$n$.}
\end{remark}

\medskip

\section{D-algebraic} \label{s:proof-dalg}
The previous sections gave us information about the parity of cogrowth sequences.
We first prove Lemma~\ref{l:dalglemma} where the parity information is enough to
conclude that a sequence is not D-algebraic.  We then deduce Theorem~\ref{t:dalg}.

\smallskip

\subsection{Proof of Lemma~\ref{l:dalglemma}} \label{ss:proof-dalg-lemma}

\smallskip

Let $\Lambda(t) = \sum \lambda_nt^n$, and suppose that $\Lambda$ satisfies an algebraic differential equation. Then there exist positive integers $C$ and $D$ together with a finite family of polynomials $\{\Pi_{c, d} \}_{0 \leq c \leq C, 0 \leq d \leq D}$, not all zero, such that for all $n$
\[
\sum_{c, d} \. \sum_{i_1 + \cdots + i_d = n - c}  \. \Pi_{c, d}(i_1, \dots, i_d) \. \lambda_{i_1}\ts \cdots \ts \lambda_{i_d} \.  = \. 0.
\]

Note that this sum has repeated terms, so e.g.\ \ts $\lambda_3\lambda_7$ \ts and \ts
$\lambda_7\lambda_3$ \ts are counted separately. We recast this as a sum over partitions:
\begin{equation}\label{eq:adee}
\sum_{c, d} \. \sum_{\substack{\nu \vdash (n - c) \. : \. |\nu| = d}} \. \Gamma_{\nu, n} \. \lambda_1 \cdots \lambda_d \. =  \. 0 \qquad
\text{for all} \ n,
\end{equation}
where \ts $\Gamma_{\nu , n} $ \ts are sums of the corresponding \ts $\Pi_{c, d}$.

Denote by \ts $v_2(x)$ \ts the largest power of~$2$ dividing~$x$.
Take some $\mu$ such that \ts $v_2(\Gamma_{\mu}, n)$ \ts is minimized.  This is always possible
because not all \ts $\Gamma_\nu$ \ts are zero, since the ADE is trivial otherwise.
If there are ties, then we pick the one where $c$ is minimal.

Let \ts $ V = v_2(\Gamma_{\mu}, n)$, and let \ts $\ell$ = $\ell(\mu)$.
By the assumption of our lemma, there exist distinct indices \ts
$n_{\alpha_1}, \dots, n_{\alpha_\ell}$, such that \ts
$n_{\alpha_i} \equiv \mu_i$ \ts modulo \ts $2^{V + 1}$. Furthermore, we can assume that all of these indices are greater than $N(C, D)$ as defined in condition~$(4)$.

We claim that this contradicts \eqref{eq:adee}. Indeed, consider the equality modulo~$2^{V + 1}$.  Letting \ts $\nu = \{n_{\alpha_1}, \dots, n_{\alpha_\ell}\}$, by the assumption we get that \ts $ V = v_2(\Gamma_{\nu}, n)$.  Since all \ts $\lambda_{n_{\alpha_i}}$ \ts are odd,  this particular term will have \ts $v_2 = V$.

 Any term with lower $c$ will have \ts $v_2(\Gamma, n) > V$, so we can ignore those terms in \eqref{eq:adee}. On the other hand, any other term besides $\nu$ will have \ts $v_2(\Gamma, n) \geq V$, and by condition $(4)$ at least one of the $\lambda_i$ is even, meaning such terms will also have \ts $v_2(\Gamma, n) > V$.

 Thus the left-hand side of \eqref{eq:adee} has exactly one term which is not congruent to zero modulo $2^{V + 1}$, a contradiction. Hence our sequence cannot be D-algebraic. \qed
%\end{proof}

\smallskip

\subsection{Proof of Theorem~\ref{t:dalg}} \label{ss:proof-dalg-main}
Suppose we have a polynomial $f$ satisfying the conditions  prescribed in Conjecture \ref{conj:sparse}. Construct $A_1, \dots, A_k$ and $P, Q, R$ as in the proof of Theorem~\ref{t:maintheorem}. Suppose for the sake of contradiction that $\cog_\mathcal{S}(n)$ and $\cog_\mathcal{T}(n)$ are both D-algebraic.

Now, let $\mathcal{W}$ be the set of cogrowth words of the form
\[
P W_1 Q W_2 P^{-1} W_3 Q^{-1} W_4\ts,
\]
where \ts $W_i$ \ts are words in \ts $\{A_1^{\pm 1}, \dots, A_k^{\pm 1}\}$.
Define \ts $\omega_n$ \ts to be the number of words in \ts $\mathcal{W}$ \ts of length~$n$.

Lemma \ref{subword} shows that the evaluations of $W_1$ and $W_3$ are the same, and are equal to the inverse of the evaluations of $W_2$ and $W_4$. Also, there must be a root $\bx = (x_1, \dots, x_k)$ of~$f$, such that the net number of \ts $A_i$'s \ts in
$W_1$ is equal to~$x_i$, for all \ts $i \in [k]$. The same must be true (up to minus sign) for \ts $W_2, W_3, W_4$.

We now proceed to make one more modification of our matrices. We expand $P$ and $Q$ by adding~$k$
copies of a $5\times 5$ matrix \ts $I_5 + E_{13}$ \ts and \ts $I_5 + E_{23} + E_{45}$, respectively:
{\small
\[
P \. \gets \. P \oplus^k   \begin{bmatrix}
1 & 0 & 1 & 0 & 0\\
0 & 1 & 0 & 0 & 0\\
0 & 0 & 1 & 0 & 0\\
0 & 0 & 0 & 1 & 0 \\
0 & 0 & 0 & 0 & 1 \\
\end{bmatrix}
\qquad \text{and} \qquad
Q \. \gets \.  Q \oplus^k \begin{bmatrix}
1 & 0 & 0 & 0 & 0\\
0 & 1 & 1 & 0 & 0\\
0 & 0 & 1 & 0 & 0\\
0 & 0 & 0 & 1 & 1 \\
0 & 0 & 0 & 0 & 1 \\
\end{bmatrix} \
\]
}
Then, for each $j$, create two versions of $A_j$. One will be \ts $A \oplus I_{5k}$,  called
the \textit{neutral version}. The other will be
\[
%{\text{old } }
 A_j \oplus I_{5(j - 1)} \oplus\begin{bmatrix}
1 & 0 & 0 & 0 & 0\\
0 & 1 & 0 & 0 & 0\\
0 & 0 & 1 & 1 & 0\\
0 & 0 & 0 & 1 & 0 \\
0 & 0 & 0 & 0 & 1
\end{bmatrix} \oplus I_{5( k - j)}\ts,
\] called the \textit{positively charged version}.
Symmetrically, there will also be a \emph{neutral} and
\emph{negatively charged version} of $A_i^{-1}$.

We have added a \ts $5k\times 5k$ \ts sub-block to each of the matrices in our generating set. Call this sub-block the \emph{new parts} of the matrix. Also let the \emph{net charge} of a word be the number of positively charged $A_i$'s minus the number of negatively charged $A_i$'s.

Let $\mathcal{W}'$  be the set of cogrowth words of the form
\[
P W_1 Q W_2 P^{-1} W_3 Q^{-1} W_4\ts,
\]
where \ts $W_i$ \ts are words in \ts $\{A_1^{\pm 1}, \dots, A_k^{\pm 1}\}$ together with their charged versions.

\begin{lemma}
A word in $\mathcal{W}'$ will be cogrowth if and only if it corresponds to a word in $\mathcal{W}$ in which $W_1$ through $W_4$ all have net charges of 0.
\end{lemma}
\begin{proof}
Suppose that the words $W_1$ through $W_4$ have charges $c_1$ through $c_4$. Then the new part of $W_i$ is
{\small
\[
\begin{bmatrix}
1 & 0 & 0 & 0 & 0\\
0 & 1 & 0 & 0 & 0\\
0 & 0 & 1 & c_i & 0\\
0 & 0 & 0 & 1 & 0 \\
0 & 0 & 0 & 0 & 1
\end{bmatrix}.
\]
}This means that the new part of the whole word can be computed to be
{\small \[
\begin{bmatrix}
1 & 0 & 0 & c_1 + c_2& c_1 - c_2 - c_3\\
0 & 1 & 0 & c_2 + c_3 & -c_2 - c_3\\
0 & 0 & 1 & c_1 +c_2 + c_3 + c_4 & -c_2 - c_3\\
0 & 0 & 0 & 1 & 0 \\
0 & 0 & 0 & 0 & 1 \\
\end{bmatrix}.
\]
}
This gives a cogrowth word if and only if \ts $c_1 = c_2 = c_3 = c_4 = 0$, as desired.
\end{proof}

\smallskip

Denote by \ts $\gamma_n$ \ts be the number of charged words which are cogrowth words, so we have \ts $\gamma_n \geq \omega_n$.
One can think of this as giving a weight to each of the words in \ts $\mathcal{W}$ \ts counting how many ways we can assign charges so that each of the \ts $W_i$ \ts has net charge zero. Since we can always neutrally charge all the $A_i$'s every word has weight at least 1. If this word is the minimal word for some root, then that is the only choice; otherwise there will be many.

Let us assign charges to the $A_i$'s in~$W_1$. Without loss of generality we can assume that \ts $x_i \geq 0$. Since there are \ts $v + x_i$ \ts instances of $A_i$  and $v$ instances of $A_i^{-1}$, there are
\[
\sum_{u = 0}^v \. \binom{v + x_i}{u} \binom{v}{u} \. = \. \binom{2v + x_i}{v}
\]
ways of doing this.  We charge $u$ each of the positive and negative ones.
It can be shown (see e.g.\ in \cite[Exc.~1.6]{Sta99}), that \ts $\binom{2v + x_i}{v}$ \ts is odd
only if there exists some positive integer $d$ such that
\begin{equation}\label{kummer}
2^d - x_i \. \leq v \. \leq \. 2^d.
\end{equation}
This implies that for a fixed~$x_i$, there will be an even number of ways of assigning charge for a set of $v$'s having density~$1$.
In particular, for there to be an odd weight on a word, we need \eqref{kummer} to hold for all $W$'s and $x$'s. That implies
\begin{equation} \label{eq:closetopowers}
|n - 4 - e|  \. \leq \. 4|\bx|\ts,
\end{equation}
where $e$ is the sum of at most $4k$ powers of~$2$. Note that we also have \ts $n - 4 \geq 4|\bx| + 4$.

Define the sequence
\[
\lambda_n \, = \, \frac{1}{2^{39}}\big(\cog_\mathcal{T}(8n + 5) \. - \. \cog_\mathcal{S}(8n + 5)\big).
\]
Then by Lemma~\ref{l:modulos}, $\{\lambda_n\}$ is a sequence of integers which is congruent to $\gamma_{2n}$ modulo~$2$. By assumption, the GF for $\{\lambda_n\}$ is D-algebraic.  We claim  that this contradicts Lemma~\ref{l:dalglemma}.

Indeed, let \ts $n_i = |\rho_i|/2$.
Conditions $(1)$, $(2)$ and $(3)$ of Lemma~\ref{l:dalglemma} follow from the assumptions of Theorem~\ref{t:dalg} and Corollary~\ref{c:otherdirection}. Therefore \ts $\{\lambda_n\}$ cannot be D-algebraic. And condition $(4)$ of Lemma~\ref{l:dalglemma} follows from the above computation plus assumption (4) of Conjecture~\ref{conj:sparse}.  As subsequences of D-algebraic sequences along arithmetic progressions are also D-algebraic, we can conclude that at least one of $\cog_\mathcal{S}$ and $\cog_\mathcal{T}$ is not D-algebraic. \qed

\bigskip

\section{Final remarks and open problems}\label{s:finrem}

\subsection{Grappling with undecidability}\label{ss:finrem-undecide}
To further understand the meaning of our Main Theorem~\ref{t:maintheorem},
we state the following corollary:

\begin{cor}\label{c:ZFC}
For some integer \ts $m \leq 9.6 \cdot 10^{85}$, there are symmetric generating
sets \ts $\CS$ \ts and \ts $\CT$ \ts of the unitriangular group \ts $\UT(m,\zz)$, such that the
following problem is independent of ZFC{} {}\footnote{We chose ZFC to make the statement
more accessible.  The proof naturally extends to any system of axioms.}~:
\[
\forall \ts n\in \nn \ : \
\cog_{\mathcal{S}}(n) \ts \equiv \ts \cog_{\mathcal{T}}(n) \ \mod 2^{40}\ts.
\]
\end{cor}

The corollary follows from a standard diagonalization argument
(see e.g.\ \cite[p.~212]{Poo14}).  Here is another corollary which
is even easier, but perhaps more suggestive.

For a matrix \ts $M=(m_{ij})$, denote  \ts $\phi(M):=\sum_{ij} \ts |m_{ij}|$ \ts
the total sum of absolute values of the entries.  Similarly, denote by \ts
$\phi(\CS) := \sum_{M\in \CS} \phi(M)$ \ts the \emph{size of $\CS$}.  The following
corollary follows from basic results on computability:

\begin{cor}\label{c:Tow}
For some integer \ts $m \leq 9.6 \cdot 10^{85}$, there are symmetric generating
set \ts $\CS$ \ts and \ts $\CT$ \ts of the unitriangular group \ts $\UT(m,\zz)$,
such that
\[
\exists \ts n\in \nn \ : \
\cog_{\mathcal{S}}(n) \ts \not\equiv \ts \cog_{\mathcal{T}}(n) \ \mod 2^{40}\ts,
\]
but the first time the inequality holds is for \ts $n\ts > \ts
\Tow(\Tow(\Tow(\phi)))$,\footnote{We stopped at three towers for clarity.  We
could just as well have written \ts $\Tow(\phi)$ \ts of towers, for example.}
where \ts $\phi:=\phi(\CS)+\phi(\CT)$.
\end{cor}

Here \ts $\Tow(k)$ \ts is the tower of $2$'s of length~$k$.  While a single
tower is unusual but does occur for natural combinatorial problems,
see e.g.\ \cite{Gow97,HNP21}, the iterated towers get us close to the edge of
human imagination.

In the context of cogrowth sequences, we can only think of \cite{Moo13} which
proves a single tower lower bound on the size of the F\o{}lner sets for the Thompson's
group~$F$.  This does not refute the conjecture that $F$ is nonamenable
(cf.~\cite[$\S$5.4]{Sap11}), but suggests that the proof would be rather involved.
We refer to a curious numerical investigation of the cogrowth sequence \cite{PG19}
(see also \cite{HHR15}), strongly suggesting nonamenability.

\subsection{Unitriangular group}\label{ss:finrem-Uni}
Jennings famously proved in~\cite{Jen55} (see also~\cite{GW17}), that every
torsion-free nilpotent group is a subgroup of the unitriangular group
\ts $\UT(m,\zz)$ \ts for some~$m$.  This explains why we chose to work
with the unitriangular group towards Kontsevich's question for nilpotent
groups.  In fact, this can be stated formally:  if the analogue of
Theorem~\ref{t:maintheorem} holds for \emph{some} \ts nilpotent group
and its families of generating sets, then the ``using multiple copies of
extra generators'' trick used in~$\S$\ref{ss:main-proof-three} one can
still obtain the first part of Theorem~\ref{t:maintheorem}.  

\subsection{Heisenberg group}\label{ss:finrem-Heis}
For the Heisenberg group \ts $H_1=\UT(3,\zz)$ \ts with natural generators,
the first 71 terms were computed by Pantone, see
\cite[\href{https://oeis.org/A307468}{A307468}]{OEIS}.  His analysis
suggests that there are no lower order algebraic differential equation
(ADE) for the cogrowth series.  We conjecture that this cogrowth series
is not D-algebraic.  Thus, in particular, it is non-D-finite and not
a diagonal.

Continuing the discussion of Stoll's example in~$\S$\ref{ss:intro-hist},
there is a deeper reason why $H_1$ has simpler structure than the higher
Heisenberg group $H_2\ssu \UT(4,\zz)$, see \cite{NY18}. In fact, from metric
geometry point of view, group $H_2$ is the ``most distorted'' relative
to the abelian group, see \cite{Naor18}.  Additionally, every equation
is decidable in $H_1$ \cite[$\S$2.2]{DLS15}, and there are relatively
few distinct words \cite{GL22}.  Thus, if one is looking for
a conceptual proof of non-D-finiteness in a smaller example,
perhaps $H_2$ or $\UT(4,\zz)$ is a better place to start than~$H_1$.

\subsection{Dependence on the generators}\label{ss:finrem-dep}
A deep problem for cogrowth series is whether their properties
depend on the generating set.  For D-finiteness we have a partial
answer: they do not for free groups and amenable groups
of superpolynomial growth (see $\S$\ref{ss:intro-hist}).
We conjecture that they do not for virtually nilpotent group
as well.  We are at loss what happens to general nonamenable
groups, but that's where we would look for counterexamples.

\subsection{Abelian groups}\label{ss:finrem-abelian}
Kuksov's Theorem~\ref{t:Kuksov} holds for general abelian groups.  We found
an alternative proof using binomial sums, which implies a stronger
statement: that the cogrowth series is always a diagonal of
an $\nn$-rational function, see \cite{GP14}.
It would be interesting to extend Theorem~\ref{t:Kuksov} to
other tame classes of group.  We conjecture that the cogrowth
series for a virtually abelian group is always a diagonal
of a rational function. Thus, in particular, it is D-finite.

\subsection{Christol's conjecture}\label{ss:finrem-Christol}
There is a healthy debate in the literature about the validity of Christol's
Conjecture~\ref{conj:Christol}. A large number of potential counterexamples were
suggested by Christol himself and his coauthors \cite{B+13,Chr14}.  A few
of these were recently refuted, i.e.\ shown to be diagonals of rational
functions \cite{AKM20,BY22}.  It would be most exciting if there is an
uncomputability result analogous to Theorem~\ref{t:diag} in this setting.

\subsection{Explicit construction}\label{ss:finrem-explicit}
The construction of generating sets in Corollary~\ref{c:ZFC} can be made explicit
if one uses an explicit construction of a Diophantine equation whose solution is
independent of~ZFC. This equation, in principle, can be obtained from an
explicit construction of a Turing machine whose halting is independent
of~ZFC, see~\cite{YA16} and follow the approach in~\cite{CM14}.
We would be curious to see the resulting numerical bounds on
the size of the resulting generating sets.

\vskip.8cm

\subsection*{Acknowledgements}
We are grateful Art\"em Chernikov, Gilles Christol, Pierre de~la~Harpe, 
Mark van Hoeij, Boris Moroz and Michael Stoll for discussions and helpful comments.
Special thanks to Boris Adamczewski and Jason Bell for telling us
about their paper~\cite{AB13}, to Andrew Marks for pointing out a
gap in our original argument, and to  Yuri Matiyasevich for help
with the references.  The first author was partially
supported by the~NSF.

\vskip.3cm

\begin{center}\begin{minipage}{13cm}%
{\em
This paper was finished soon after the death of Mark Sapir.  Over the years,
the first author had many conversations with Mark, whose wit and generosity
were delightful and educational.  We dedicate this paper to his memory. }
\end{minipage}\end{center}

\end{document}